\documentclass[3p,preprint]{elsarticle}

\usepackage{amsfonts}
\usepackage{amsmath}
\usepackage{geometry}

\usepackage{algorithmicx}
\usepackage{algorithm}
\usepackage{algpseudocode}
\usepackage{color}
\usepackage{relsize}
\usepackage{csquotes}
\usepackage{amssymb}

\newtheorem{theorem}{Theorem}[section]

\newtheorem{lemma}[theorem]{Lemma}

\newtheorem{proposition}[theorem]{Proposition}
\newtheorem{remark}[theorem]{Remark}

\newenvironment{proof}[1][Proof]{\noindent\textbf{#1.} }{\ \rule{0.5em}{0.5em}}

\biboptions{sort&compress}

\begin{document}

\bigskip

\bigskip 
\begin{frontmatter}

\title{A $C^{\infty}$ rational quasi-interpolation operator for functions with jumps without the Gibbs phenomenon}

\author[address-CS]{Francesco Dell'Accio\corref{corrauthor}}
\cortext[corrauthor]{Corresponding author}
\ead{francesco.dellaccio@unical.it}

\author[address-CS]{Francesco Larosa}
\ead{francesco.larosa@unical.it}

\author[address-CS]{Federico Nudo}
\ead{federico.nudo@unical.it}

\author[address-MC]{Najoua Siar}
\ead{s.najoua@umi.ac.ma}

\address[address-CS]{Department of Mathematics and Computer Science, University of Calabria, Rende (CS), Italy}

\address[address-MC]{Department of Mathematics, University Moulay Ismail, Meknes, Morocco}

\begin{abstract}
The study of quasi-interpolation has gained significant importance in numerical analysis and approximation theory due to its versatile applications in scientific and engineering fields. This technique provides a flexible and efficient alternative to traditional interpolation methods by approximating data points without requiring the approximated function to pass exactly through them. This approach is particularly valuable for handling jump discontinuities, where classical interpolation methods often fail due to the Gibbs phenomenon. These discontinuities are common in practical scenarios such as signal processing and computational physics. In this paper, we present a $C^{\infty}$ rational quasi-interpolation operator designed to effectively approximate functions with jump discontinuities while minimizing the issues typically associated with traditional interpolation methods.
\end{abstract}


\end{frontmatter}

\section{Introduction}
 A fundamental problem in numerical analysis involves reconstructing an unknown continuous function $f$, defined on an interval $[a,b]$, given only its evaluations on a set of nodes. Classical interpolation and quasi-interpolation methods can be employed to address this problem. The classical interpolation method seeks an approximating function in a fixed vector space of finite dimension that passes precisely through the known data, by solving a linear system. In contrast, the quasi-interpolation method provides a more flexible approach by avoiding this strict constraint and instead requiring only polynomial reproduction up to a certain degree~\cite{Wang:2004:QIW,Buhmann:2022:QI}. The study of quasi-interpolation has become increasingly essential in numerical analysis and approximation theory due to its versatile use across various scientific and engineering applications~\cite{Ma:2009:ATT,Barrera:2013:ITA,Barrera:2019:QIB,Sun:2022:ACI,Barrera:2022:ANC,Ortmann:2024:HAQ,Arandiga:2024:EAW}. Introduced by Schoenberg in~\cite{Schoenberg:1946:CTTA,Schoenberg:1946:CTTB}, quasi-interpolation offers several advantages: it is
efficient and relatively easy to formulate for scattered and meshed nodes, and it is not necessary to solve any linear system of equations~\cite{Arandiga:2024:EAW,Arandiga:2024:NUW}.

Another critical challenge in numerical analysis is approximating an unknown function on an interval $[a,b]$ with jump discontinuities (or discontinuities of the first kind). These functions are common in various practical applications, such as signal processing, computational physics, and data analysis~\cite{Cockburn:1998:ENO}. Furthermore, these functions describe many real-world phenomena, such as shock waves in fluid dynamics and abrupt changes in financial markets. Accurately approximating these functions is essential for creating models that can effectively predict and simulate real-world behaviors. In this context, classical interpolation methods often struggle to effectively manage discontinuities, leading to the Gibbs phenomenon~\cite{Gibbs:1898:FSS}. This phenomenon is characterized by oscillations near discontinuities, which compromises the accuracy of the approximation. Conversely, the flexibility of quasi-interpolants is advantageous in mitigating this issue~\cite{Buhmann:2024:NMF,Arandiga:2024:EAW}. Recent contributions have further highlighted the effectiveness and versatility of quasi-interpolation techniques, especially in the context of low-degree spline constructions and non-standard differential settings; see for instance~\cite{guessab2011error, barrera2023lownew}. The method proposed in~\cite{Arandiga:2024:EAW} employs $C^1$-continuous cubic quasi-interpolation schemes in Bernstein-Bézier form to approximate functions with jump discontinuities. This approach not only enhances convergence but also mitigates the Gibbs phenomenon by applying a WENO (Weighted Essentially Non-Oscillatory) technique to the scheme's parameters. In~\cite{Amat:2021:OCS}, a new nonlinear technique specifically designed to be free of the Gibbs phenomenon has been developed. This technique introduces a nonlinear procedure that adjusts the behavior of the splines near discontinuities, at the cost of losing $C^2$ regularity. To overcome this issue, in~\cite{Amat:2022:ACO}, a new reconstruction has been introduced, by effectively removing the high-frequency oscillations while retaining the sharp features of the data, yielding a spline that is both smooth and accurate. The interested reader can refer to~\cite{Foucher:2009:QSQ,Gao:2014:AQI} and the references therein in order to obtain more information about the quasi-interpolation operator using spline functions, which are piecewise polynomials with a bounded differentiability class $C^k$, $k \in \mathbb{N}$.

In this paper, we introduce a new rational quasi-interpolant by employing the multinode Shepard functions~\cite{DellAccio:2018:ROF,DellAccio:2019:RCM}, which have been widely used in several applications, see~\cite{DellAccio:2021:NDS,DellAccio:2021:SPE,DellAccio:2023:OTI,Cavoretto:2024:NCS,DellAccio:2024:MSM}. 
More precisely, we subdivide the interval $[a,b]$ into several subintervals. For each subinterval, we compute the multinode Shepard function based on pairwise distinct points within that interval, as well as the local interpolation polynomial using the available data. The multinode Shepard functions are then used to blend the local interpolation polynomials, obtaining a global interpolant. This method offers significant advantages over spline functions. Notably, the multinode Shepard functions are $C^{\infty}$ functions, ensuring that the quasi-interpolant inherits this high level of differentiability. Moreover, since the multinode Shepard functions mimic the behavior of characteristic functions, this choice greatly provides an effective means of approximating functions with jump discontinuities. This approach significantly mitigates the issues typically associated with traditional interpolation methods.

The paper is organized as follows. In Section~\ref{sec2}, we introduce the rational quasi-interpolation operator by assuming to know the intervals where the discontinuities are located. In Section~\ref{sec3} we study the error bound of this operator on the subintervals where the function $f$ is continuous. In Section~\ref{sec4}, we present numerical experiments demonstrating the efficacy and accuracy of our method.

\section{A $C^{\infty}$ rational quasi-interpolation operator}
\label{sec2}
Let $f$ be an unknown function defined on an interval $[a,b]$ and suppose to know only its evaluations on a set of $n+1$ nodes
\begin{equation}
    X_n=\left\{x_0,\dots,x_n\right\}, \quad a=x_0<\dots<x_n=b.
    \label{Xn}
\end{equation}
We denote by 
\begin{equation*}
    S=\left\{s_{1},\dots,s_{m}\right\}
\end{equation*}
the set of discontinuity points of the function $f$ and we assume that the only information about them is that they occur somewhere between two consecutive nodes of the set $X_n$. More precisely, we assume that
\begin{equation*}   s_{\ell}\in\left(x_{\alpha_{\ell}},x_{\alpha_{\ell}+1}\right),
\end{equation*}
where $\alpha_{\ell}\in \mathbb{N}$, $\ell=1,\dots,m$ and $0<\alpha_{1}<\dots<\alpha_{m}<n-1$.
We set $x_{\alpha_0}=x_0$,  $x_{\alpha_{m+1}}=x_n$ and, by assuming that
\begin{equation}\label{alphapiu1}
    \alpha_{\ell}+1<\alpha_{\ell+1}, \quad \ell=1,\dots,m-1, 
\end{equation}
we consider the following intervals 
\begin{equation}\label{continuityints}
 I_0=\left[x_{\alpha_0},x_{\alpha_1}\right], \quad I_{\ell}=\left[x_{\alpha_{\ell}+1},x_{\alpha_{\ell+1}}\right], \quad j=1,\dots,m-1, \quad I_{m}=\left[x_{\alpha_m+1},x_{\alpha_{m+1}}\right].
 \end{equation}
We aim to define a rational quasi-interpolation operator by blending local polynomial interpolants of degree $d \in \mathbb{N}_0$ using multinode Shepard functions~\cite{DellAccio:2018:ROF}. To this aim, we denote the lengths of the intervals $I_{\ell}$ by 
\begin{equation*}
    h_{\ell}=\left\lvert I_{\ell} \right\rvert, \quad \ell=0,\dots,m, 
\end{equation*}
and we set
\begin{equation}\label{hmin}
    h^{\min}=\min_{\ell=0,\dots,m} h_{\ell}.
\end{equation}
By assumption~\eqref{alphapiu1} it results $h^{\min}>0$. Further, we set 

\begin{eqnarray}\label{hXnmax}
    && h_{X_n}^{\max} = \max_{\substack{0 \leq i \leq n-1 \\ s_\ell \notin \left(x_i, x_{i+1}\right) \, \forall \ell \in \{1, \dots, m\}}} \left\lvert x_i - x_{i+1} \right\rvert \\
     && h_{X_n}^{\min} = \min_{\substack{0 \leq i \leq n-1 \\ s_\ell \notin \left(x_i, x_{i+1}\right) \, \forall \ell \in \{1, \dots, m\}}} \left\lvert x_i - x_{i+1} \right\rvert\label{hXnmin}
\end{eqnarray}
and assume that
\begin{equation}\label{hi}
    h^{\min}\geq  h_{X_n}^{\max}.
\end{equation}
For any $r>0$ and $x\in [a,b]$ we set  
\begin{equation*}
    \mathcal{I}_{r,x}=\left\{I\subset \bigcup_{\ell=0}^m I_{\ell} \, :  \, \text{ $I$ is a closed interval, } \, \lvert I \rvert = r, \, x\in I \right\},
\end{equation*}
and for any $d \in \mathbb{N}_0$ we also set 
\begin{equation} \label{rd}
 r_0=\inf\left\{r\in\mathbb{R}_{+} \, :\, \left(r\leq h^{\min}\right) \wedge \left( \forall x \in  \bigcup_{\ell=0}^m I_{\ell} \quad \exists I \in \mathcal{I}_{r,x} \, s.t. \, \mathrm{card}\left(I\cap X_n\right)\geq 1 \right) \right\},
    \end{equation}
where $\mathrm{card}(\cdot)$ is the cardinality operator. Then we have 
\begin{theorem}\label{theorem:hxnmax}
Under the assumption~\eqref{hi} $r_0$ exists and $r_{0} \in \mathbb{R}_{+}$.
\end{theorem}
\begin{proof}
Let $\ell\in\{0,\dots,m\}$ be such that $I_\ell= \left[a_\ell, b_\ell\right]$.
For each $x \in I_\ell \setminus \left\{b_{\ell}\right\}$, if $x \notin X_n$ we denote by $\delta_x=\left[x_{i},x_{i+1}\right]$ the unique interval bounded by two consecutive nodes that contains $x$; 
if instead $x \in X_n$, then $x=x_i$ for some $i$ and we set $\delta_x=\left[x_i,x_{i+1}\right]$. To prove the existence of $r_0>0$, we show that it is possible to construct a sequence of intervals $U_{\ell,1}, \dots, U_{\ell,n_{\ell}}$ such that, for any $j=1,\dots,n_{\ell}$,
 \begin{equation*}
       \left \lvert U_{\ell,j} \right \rvert = h^{\max}_{X_n}, \qquad \operatorname{card}\left(U_{\ell,j}\cap X_n\right)\geq 1, 
   \end{equation*}
   and
   \begin{equation*}
       \bigcup\limits_{j=1}^{n_{\ell}} U_{\ell,j} = I_{\ell}.
   \end{equation*}
From $h^{\min}\ge h^{\max}_{X_n}$ it results
        \begin{equation}
            a_\ell\in \left[a_\ell,a_\ell+h^{\max}_{X_n}\right]\subset I_\ell
        \end{equation}
        and we set 
        \begin{equation}
U_{\ell,1}=\left[a_{\ell,1},b_{\ell,1}\right]=\left[a_\ell,a_\ell+h^{\max}_{X_n}\right].
        \end{equation}
Then
   \begin{equation*}
        \left \lvert U_{\ell,1} \right \rvert =h^{\max}_{X_n}, \qquad  \operatorname{card}\left(U_{\ell,1}\cap X_n\right)\geq 1,  
   \end{equation*}
   so that, if $b_{\ell,1}=b_{\ell}$, we have got the sequence.
Otherwise, either $b_{\ell,1}+h^{\max}_{X_n}> b_\ell$ or $b_{\ell,1}+h^{\max}_{X_n}\leq b_\ell$.\\
 If $b_{\ell,1}+h^{\max}_{X_n}> b_\ell$, we set 
            \begin{equation*}
U_{\ell,2}=\left[a_{\ell,2},b_{\ell,2}\right]=\left[b_{\ell}-h^{\max}_{X_n},b_{\ell}\right]. 
            \end{equation*}
Then
 \begin{equation*}
   \left \lvert U_{\ell,2} \right \rvert =h^{\max}_{X_n}, \qquad   \operatorname{card}\left(U_{\ell,2}\cap X_n\right)\geq 1,
   \end{equation*}
and $U_{\ell,1} \cap U_{\ell,2} \neq \emptyset$ since $a_{\ell,2}=b_{\ell}-h^{\max}_{X_n} <  b_{\ell,1}<b_{\ell}=b_{\ell,2}$, and we have got the sequence.\\
If, instead, $b_{\ell,1}+h^{\max}_{X_n}\le b_\ell$, we set 
    \begin{equation*}
U_{\ell,2}=\left[a_{\ell,2},b_{\ell,2}\right]=
\left[\min \delta_{b_{\ell,1}},\min \delta_{b_{\ell,1}}+h^{\max}_{X_n}\right]. 
    \end{equation*}
Then 
   \begin{equation*}
   \left \lvert U_{\ell,2} \right \rvert =h^{\max}_{X_n}, \qquad   \operatorname{card}\left(U_{\ell,2}\cap X_n\right)\geq 1,
   \end{equation*}
   and
   \begin{equation*}
       U_{\ell,1}\cap U_{\ell,2} \neq \emptyset, \qquad \left\lvert U_{\ell,2}\setminus U_{\ell,1} \right\rvert > 0.
   \end{equation*}
If $b_{\ell,2}+h^{\max}_{X_n} \geq b_{\ell}$, we have got the sequence. Otherwise, by assuming that we have already got the sequence of intervals $U_{\ell,1}, \dots, U_{\ell,k-1}$, $k\geq 3$, such that, for any $j=1,\dots,k-1$, 
   \begin{equation*}
       \left \lvert U_{\ell,j} \right \rvert =h^{\max}_{X_n}, \qquad \operatorname{card}\left(U_{\ell,j}\cap X_n\right)\geq 1,  
   \end{equation*}
   and
   \begin{equation*}
       U_{\ell,j-1}\cap U_{\ell,j} \neq \emptyset, \qquad \left\lvert U_{\ell,j}\setminus U_{\ell,j-1} \right\rvert > 0,
   \end{equation*} 
for any $j=2,\dots,k-1$, the previous step shows that it is possible to construct another interval of the sequence, which will either allow us to obtain the sequence $U_{\ell,1},\dots,U_{\ell,n_{\ell}}$ or require another step. The process certainly ends since $I_{\ell} \cap X_n$ is finite. 
Since $h^{\max}_{X_n}$ satisfies the condition
    \begin{equation*}
\left(h^{\max}_{X_n}\leq h^{\min}\right) \wedge \left( \forall x \in  \bigcup_{\ell=0}^m I_{\ell} \quad \exists I \in \mathcal{I}_{r,x} \, s.t. \, \mathrm{card}\left(I\cap X_n\right)\geq 1 \right)
    \end{equation*}
    while any positive number $r<\frac{h^{\min}_{X_n}}{2}$ does not, and since $\ell$ was chosen arbitrarily, the theorem follows.
\end{proof}

The proof of Theorem~\ref{theorem:hxnmax} shows in constructive way the existence, for any $\ell=0,\dots,m$, of intervals $U_{\ell,1},U_{\ell,2}, \dots, U_{\ell,n_{\ell}}$ such that
    \begin{eqnarray*}
        &&\left\lvert U_{\ell,j}\right\rvert =h^{\max}_{h_{X_n}}, \quad \mathrm{card}\left(  U_{\ell,j} \cap X_n\right) \geq 1, \quad \forall j=1,\dots,n_{\ell}\\    
        && U_{\ell,j} \cap U_{\ell,j+1} \neq \emptyset, \quad \forall j=1,\dots,n_{\ell}-1
    \end{eqnarray*}
    and
    \begin{equation*}
        I_{\ell}=\bigcup_{j=1}^{n_{\ell}} U_{\ell,j}.
    \end{equation*}
Similar coverings are possible for any $r\geq r_0$, as stated in the following theorem.
\begin{theorem}
    For any $\ell=0,\dots,m$ there exist intervals $U_{\ell,1},U_{\ell,2}, \dots, U_{\ell,n_{\ell}}$ such that
    \begin{eqnarray*}
        &&\left\lvert U_{\ell,j}\right\rvert =r_0, \quad \mathrm{card}\left(  U_{\ell,j} \cap X_n\right) \geq 1, \quad \forall j=1,\dots,n_{\ell}\\    
        && U_{\ell,j} \cap U_{\ell,j+1} \neq \emptyset, \quad \forall j=1,\dots,n_{\ell}-1
    \end{eqnarray*}
    and
    \begin{equation*}
        I_{\ell}=\bigcup_{j=1}^{n_{\ell}} U_{\ell,j}.
    \end{equation*}
\end{theorem}
\begin{proof}
As in the proof of the previous theorem, let $\ell\in\{0,\dots,m\}$ be such that $I_\ell= \left[a_\ell, b_\ell\right]$.
For each $x \in I_\ell \setminus \left\{b_{\ell}\right\}$, if $x \notin X_n$ we denote by $\delta_x=\left[x_{i},x_{i+1}\right]$ the unique interval bounded by two consecutive nodes that contains $x$; 
if instead $x \in X_n$, then $x=x_i$ for some $i$ and we set $\delta_x=\left[x_i,x_{i+1}\right]$. 
We set  
        \begin{equation}
U_{\ell,1}=\left[a_{\ell,1},b_{\ell,1}\right]=\left[a_\ell,a_\ell+r_0\right]
        \end{equation}
and we get $\operatorname{card}\left(U_{\ell,1}\cap X_n\right)\geq 1$. If $b_{\ell,1}=b_{\ell}$, we get the thesis. 
On the contrary, we assume that we have found a sequence of intervals $U_{\ell,1},\dots,U_{\ell,k-1}$, $k\geq 2$, such that, for any $j=1,\dots,k-1$, it results  
   \begin{equation*}
       \left \lvert U_{\ell,j} \right \rvert = r_0, \qquad  \operatorname{card}\left(U_{\ell,j}\cap X_n\right)\geq 1,  
   \end{equation*}
   and, in case of $k\geq3$, for any $j=2,\dots,k-1$ 
   \begin{equation*}
       U_{\ell,j-1}\cap U_{\ell,j} \neq \emptyset, \qquad \left\lvert U_{\ell,j}\setminus U_{\ell,j-1} \right\rvert > 0.
   \end{equation*}
Let assume that $b_{\ell,k-1}+r_0\le b_\ell$.
    If there exist $x_i \in U_{\ell,k-1}\cap X_n$ such that $x_i>a_{\ell,k-1}$, we set 
    \begin{equation*}
U_{\ell,k}=\left[a_{\ell,k},b_{\ell,k}\right]=
\left[\min \delta_{b_{\ell,k-1}},\min \delta_{b_{\ell,k-1}}+r_0\right] 
    \end{equation*}
    and we get $\mathrm{card}\left(U_{\ell,k}\cap X_n\right)\geq1$ and 
    \begin{equation*}
      U_{\ell,k-1}\cap U_{\ell,k} \neq \emptyset, \qquad \left\lvert U_{\ell,k}\setminus U_{\ell,k-1} \right\rvert > 0,  
    \end{equation*}
    since $\min \delta_{b_{\ell,k-1}}\geq x_i>a_{\ell,k-1}$.
    Otherwise $a_{\ell,k-1}$ is the only node in $U_{\ell,k-1}$ and we set 
    \begin{equation*} U_{\ell,k}=\left[a_{\ell,k},b_{\ell,k}\right]=
\left[b_{\ell,k-1},b_{\ell,k-1}+r_0\right]. 
    \end{equation*}
    We have to prove that $\mathrm{card}\left(U_{\ell,k}\cap X_n\right)\geq1$. By contradiction, let us assume that $U_{\ell,k}\cap X_n = \emptyset$. For any $x \in I_{\ell}$ let $I(x) \in \mathcal{I}_{r_0,x}$ such that $\mathrm{card}\left(I(x) \cap X_n\right) \geq 1$ and let  
    \begin{equation*}
        x_{i}=\min   \bigcup_{x \in \left(b_{\ell,k},b_{\ell}\right)} I(x) \cap X_n>b_{\ell,k}.  
    \end{equation*} 
 Then, for any $x \in \left(a_{\ell,k}, a_{\ell,k}+x_i-b_{\ell,k} \right)$ does not exist any $I \in \mathcal{I}_{r_0,x}$ such that $\operatorname{card}\left(I \cap X_n\right) \ge 1$, which contradicts the existence of $r_0$. Therefore $\mathrm{card}\left(U_{\ell,k}\cap X_n\right)\geq1$ and
 \begin{equation*}
     U_{\ell,k}\cap U_{\ell,k-1}\neq \emptyset, \qquad \left\lvert U_{\ell,k}\setminus U_{\ell,k-1}\right\rvert>0.
 \end{equation*}
If $b_{\ell,k-1}+r_0> b_\ell$, then we set 
            \begin{equation*}
U_{\ell,k}=\left[a_{\ell,k},b_{\ell,k}\right]=\left[b_{\ell}-r_0,b_{\ell}\right], 
            \end{equation*}
and $U_{\ell,k-1} \cap U_{\ell,k} \neq \emptyset$ since $a_{\ell,k}=b_{\ell}-r_0 \leq  b_{\ell,k-1}<b_{\ell}=b_{\ell,k}$. 
\end{proof}
\medskip

The numerical computation of the value $r_0$ is beyond the scope of this paper, since the computation of $h^{\max}_{X_n}$ is extremely simple and sufficient to produce a similar covering, possibly with more nodes in each interval $U_{\ell,j}$ compared to the one obtained using $r_0$. To ensure the existence of covering of $\bigcup\limits_{\ell=0}^m I_{\ell}$ by intervals $U_{\ell,j}$, each of which contains at least $d+1$ nodes, we introduce the following notations: 
\begin{eqnarray}
    && h_{X_n}^{\max,d} = \max_{\substack{0 \leq i \leq n-(d+1) \\ s_\ell \notin \left(x_i, x_{i+(d+1))}\right) \, \forall \ell \in \{1, \dots, m\}}} \left\lvert x_i - x_{i+(d+1)} \right\rvert \\
     && h_{X_n}^{\min,d} = \min_{\substack{0 \leq i \leq n-(d+1) \\ s_\ell \notin \left(x_i, x_{i+(d+1))}\right) \, \forall \ell \in \{1, \dots, m\}}} \left\lvert x_i - x_{i+(d+1)} \right\rvert 
     \end{eqnarray}
and 
\begin{equation}
    r_d=\inf\left\{r\in\mathbb{R}_{+} \, :\, \left(r\leq h^{\min}\right) \wedge \left( \forall x \in  \bigcup_{\ell=0}^m I_{\ell} \quad \exists I \in \mathcal{I}_{r,x} \, s.t. \, \mathrm{card}\left(I\cap X_n\right)\geq d+1 \right) \right\}.
\end{equation}

\begin{theorem}\label{theoremhdxnmax}
    Under the assumption 
    \begin{equation*}
        h^{\min} \geq h_{X_n}^{\max,d}
    \end{equation*}
    $r_d$ exists and $r_d \in \mathbb{R}_{+}$.
\end{theorem}
\begin{proof}
We use the same notation as that employed in the proof of Theorem~\ref{theorem:hxnmax}. To prove the existence of $r_d>0$, we show that it is possible to construct a sequence of intervals $U_{\ell,1}, \dots, U_{\ell,n_{\ell}}$ such that, for any $j=1,\dots,n_{\ell}$,
 \begin{equation*}
       \left \lvert U_{\ell,j} \right \rvert = h^{\max,d}_{X_n}, \qquad \operatorname{card}\left(U_{\ell,j}\cap X_n\right)\geq d+1, 
   \end{equation*}
   and
   \begin{equation*}
       \bigcup\limits_{j=1}^{n_{\ell}} U_{\ell,j} = I_{\ell}.
   \end{equation*}
The argument follows similar lines to those used in the proof of Theorem~\ref{theorem:hxnmax}, and relies on the fact that any interval of length $h_{X_n}^{\max, d}$ with at least one endpoint in $X_n$  must contain at least $d+1$ nodes. 
\end{proof}
\medskip

We notice that $r_{n+2}$ does not exist in any case, since $\mathrm{card}(X_n)=n+1$. We set 
\begin{equation} \label{dmax}
  d_{\max}=\max\left\{ d \,:\, r_d\text{ exists} \right\}.  
\end{equation}
By construction, the existence of $r_{d_{\max}}$ implies the existence of $r_d$ for any $d \in \left\{0,\dots,d_{\max} \right\}$. 

\bigskip

In the following, we assume that $h^{\min} \geq h^{\max,d}_{X_n}$ and therefore that $d\in \left\{0,\dots,d_{\max}\right\}$. For any $\ell=0,\dots,m$, let $U_{\ell,1},\dots,U_{\ell,n_{\ell}}$ be a covering of $I_{\ell}$ as in the proof of Theorem~\ref{theoremhdxnmax}, that is such that 
\begin{eqnarray*}
 \left\lvert U_{\ell,j}\right\rvert =h^{\max,d}_{X_n}, \quad
 \mathrm{card}\left(  U_{\ell,j} \cap X_n\right) \geq d+1, \quad U_{\ell,j}\cap U_{\ell,j+1} \neq \emptyset.
\end{eqnarray*}
We set 
\begin{equation}
N_{\ell,j}= U_{\ell,j}\cap X_n, \quad \ell=0, \dots, m, \quad j=0,\dots,n_{\ell}.
\end{equation}
Sometimes, in the following, for the sake of simplicity we write the pair of indices $(\ell,j)$ by using only one index $\iota$, that is
\begin{equation} \label{notation}
    N_{\ell,j}\rightarrow N_{\iota}, \quad U_{\ell,j}\rightarrow U_{\iota}, \quad\iota=1,\dots,M,
\end{equation}
where 
\begin{equation}
    M=\sum\limits_{\ell=0}^m \left(n_{\ell}+1\right).
\end{equation}
For any $\iota=1,\dots,M$, we denote by $p_{\iota}(f,x)$ the lowest degree polynomial interpolating on the nodes of $N_{\iota}$ and,
for any $K\in \mathbb{N}$, we introduce sets of pairwise distinct points,
\begin{equation}\label{Ciota}
    C_{\iota}=\left\{\xi^{\iota}_{\kappa} \, : \,  a_{\iota}<\xi^{\iota}_{1}<\dots<\xi^{\iota}_{K}<b_{\iota}\right\}, \quad a_{\iota}= \min U_{\iota}, \ b_{\iota}=\max U_{\iota}.
\end{equation}

By construction, it results 
\begin{equation}\label{cst}
b_{\iota}-a_{\iota}=h^{\max,d}_{X_n}, \quad \iota=1,\dots,M.  
\end{equation}
We introduce the rational operator
\begin{equation}\label{Qoperator}
    Q_{\mu}[f](x)=\sum\limits_{\iota=1}^M B_{\mu,\iota}(x)p_{\iota}(f,x), \quad \mu\in 2\mathbb{Z}_{+},
\end{equation}
where
\begin{equation}\label{multinodefunction}
B_{\mu,\iota}(x)=\frac{\prod\limits_{\kappa=1}^{K} \left\lvert x-\xi^{\iota}_{\kappa} \right\rvert^{-\mu}}{\sum\limits_{\tau=1}^M\prod\limits_{\kappa=1}^{K} \left\lvert x-\xi^{\tau}_{\kappa} \right\rvert^{-\mu}}, \quad \iota=1,\dots,M,
\end{equation}
are the univariate \textit{multinode Shepard functions} based on $\left\{C_{\iota}\right\}^M_{\iota=1}$.  
\medskip

In order to give an explanation of the method's underlying idea, we assume that $\left[a,b\right]=\left[-1,1\right]$, 
\begin{equation*}
    x_i=-1+\frac{2i}{n}, \quad n\in2\mathbb{Z}_{+}, \quad i=0,\dots,n,
\end{equation*}
and $s_1=x_{\frac{n}{2}}=0$ is the unique discontinuity point, that is $m=1$. We set $d=n/2-1$ and we have $N_1=\left\{x_0,\dots,x_{\frac{n}{2}-1}\right\}$ and $N_2=\left\{x_{\frac{n}{2}+1},\dots,x_{n}\right\}$.
For $\mu=4$, the multinode Shepard functions based on $K=10$ equispaced points $\xi^{\iota}_{\kappa}$  
\begin{equation*}
B_{4,1}(x)=\frac{\prod\limits_{\kappa=1}^{K} \left\lvert x-\xi^1_{{\kappa}} \right\rvert^{-4}}{\prod\limits_{\kappa=1}^{K} \left\lvert x-\xi^1_{{\kappa}} \right\rvert^{-4}+\prod\limits_{\kappa=1}^{K} \left\lvert x-\xi^2_{{\kappa}} \right\rvert^{-4}}, 
\quad
B_{4,2}(x)=\frac{\prod\limits_{\kappa=1}^{K} \left\lvert x-\xi^2_{{\kappa}} \right\rvert^{-4}}{\prod\limits_{\kappa=1}^{K} \left\lvert x-\xi^1_{{\kappa}} \right\rvert^{-4}+\prod\limits_{\kappa=1}^{K} \left\lvert x-\xi^2_{{\kappa}} \right\rvert^{-4}},
\end{equation*}
are plotted in Figure~\ref{MSB}. As we can see, $B_{4,1}(x)$ and $B_{4,2}(x)$ well approximate the characteristic functions of the intervals $\left[-1,0\right]$, $\left[0,1\right]$, respectively. More in general, it is well known that the multinode Shepard functions are non-negative
\begin{equation}\label{nonnegprop}
    B_{\mu ,\iota}\left(x\right) \geq 0,
\end{equation}
sum up to $1$
\begin{equation}\label{partofuni}
\sum\limits_{\iota=1}^M B_{\mu ,\iota}\left(x\right)=1,
\end{equation}
and $B_{\mu,\iota}(x)$ goes rapidly to zero outside the interval $U_{\iota}$, as shown in Figure~\ref{MSB} for the considered particular case.
 Moreover, $B_{\mu,\iota}(x)$ vanishes at all points $\xi^{\lambda}_{\kappa}$ that are not in $C_{\iota}$, that is 
\begin{equation}\label{vanishes}
    B_{\mu,\iota}(\xi^{\lambda}_{\kappa})=0, \qquad \xi^{\lambda}_{\kappa} \notin C_{\iota}.
\end{equation}
By properties~\eqref{partofuni} and~\eqref{vanishes} easily follows that $B_{\mu,\iota}$ and $B_{\mu,\lambda}$ could present oscillations in $U_{\iota}\cap U_{\lambda}$, $\iota \neq \lambda$, since the presence of the points $\xi^{\lambda}_{\kappa} \notin C_{\iota}$ and $\xi^{\iota}_{\kappa} \notin C_{\lambda}$. These oscillations can be avoided by forcing the sets $C_\iota$ and $C_{\lambda}$ to share the same points on the overlap $U_\iota \cap U_{\lambda}$.  
As a consequence, the function $B_{\mu,\iota}$ is a $C^{\infty}$ rational approximation of the characteristic function of the interval $U_{\iota}$
\begin{equation*}
   \chi_{U_{\iota}}(x) = 
\begin{cases} 
1 & \text{if } x \in U_{\iota}, \\
0 & \text{if }  x \notin U_{\iota}.
\end{cases}
\end{equation*}
Therefore, the product $B_{\mu,\iota}(x)p_{\iota}(f,x)$ results in a $C^{\infty}$ rational function which is near to $p_{\iota}(f,x)$ in the open interval $\mathring{U_{\iota}}$ and near to zero outside the interval $U_{\iota}$. When $U_{\iota-1}\cap U_{\iota}\neq\emptyset$, both polynomials $p_{\iota-1}(f,x)$ and $p_{\iota}(f,x)$  contribute to the approximant $Q_{\mu}[f](x)$ in the interval $U_{\iota-1} \cap U_{\iota}$.
\begin{figure}
    \centering   \includegraphics[width=0.5\linewidth]{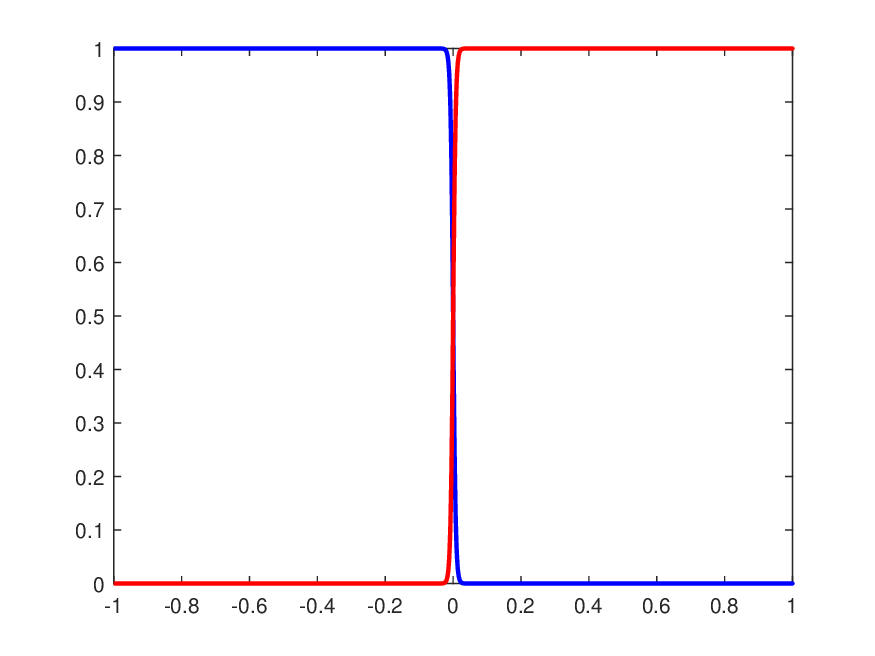}
 \caption{Multinode Shepard function $B_{4,1}$ (in blue) and $B_{4,2}$ (in red) obtained by subdividing the unit interval $\left[-1,1\right]$ in the two subintervals $\left[-1,0\right]$ and $\left[0,1\right]$ by using $K=10$ equispaced nodes in each subinterval.}
    \label{MSB}
\end{figure}

\begin{theorem}
Let $d\in \mathbb{N}_0$ be a positive integer number. Under the assumption 
 \begin{equation*}
     h^{\min}\geq h^{\max,d}_{X_n},
 \end{equation*}
the operator $Q_{\mu}[f](x)$ is a quasi-interpolation operator. 
\end{theorem}
\begin{proof}
   Let $f(x)$ be a polynomial of degree less than or equal to 
   $\displaystyle{\min_{\iota=1,\dots,M} \operatorname{deg}(p_{\iota}(f,x))}\geq d.$ Then, 
   $p_{\iota}(f,x)=f(x)$ for each $\iota=1,\dots,M$, and by~\eqref{partofuni}, it results 
\begin{equation*}
Q_{\mu}[f](x)=\sum\limits_{\iota=1}^M B_{\mu,\iota}(x)p_{\iota}(f,x)=\sum_{\iota=1}^M B_{\mu,\iota}(x) f(x)= f(x), \quad x\in[a,b].
\end{equation*}
\end{proof}

\section{Error bound}\label{sec3}
In this section, we discuss the error bound of the quasi-interpolation operator~\eqref{Qoperator} on the intervals $I_{\ell}$ where the function $f$ is continuous. For the sake of simplicity, we assume that the nodes of $X_n$ are equispaced and that the covering $\left\{U_{\ell,j}\right\}$ of each interval of continuity $I_{\ell}$ is realized by intervals of the family 
\begin{equation*}
    \left\{ \left[x_i,x_i+\left(d+1\right)\dfrac{b-a}{n}\right] \right\}_{x_i \in I_{\ell}}
\end{equation*}
each of which intersects the subsequent one in only one node. By assuming that the points $\xi^{\iota}_{\kappa}$ are equispaced, we show that it is possible to identify a proper subset $\Xi \subsetneq \bigcup\limits_{\ell=0}^m I_{\ell}$
\begin{equation}\label{Xi}
    \Xi= \bigcup\limits_{\ell=0}^m \left[\xi^{\ell}_1,\xi^{n_\ell}_K\right],
\end{equation}
where the approximation to $f(x)$ provided by $Q_{\mu}[f](x)$ follows that one provided by $p_{\iota}(f,x)$ in each interval $U_{\iota}\subset \bigcup\limits_{\ell=0}^m I_{\ell}$.
In the complementary set $\bigcup\limits_{\ell=0}^m I_{\ell} \setminus \Xi$ the accuracy of approximation wastes.

The following Lemma, which refers to the simplest non-trivial case of two multinode Shepard functions based on the points $\left\{\xi^1_{\kappa}\right\}$, $\left\{\xi^2_{\kappa}\right\}$, $\kappa=1,\dots,K$, of two consecutive intervals of equal lengths (see Fig.~\ref{xleft} and Fig.~\ref{xright}), is useful to provide bounds of the multinode Shepard function $B_{\mu,\iota}(x)$, for different relative positions of $x \in \Xi$.

\begin{figure}
    \centering
  \includegraphics[width=0.8\linewidth]{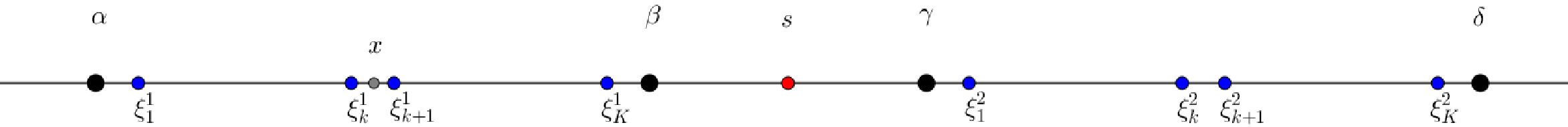}
    \caption{Plot of $K$ equispaced points (in blue) for each of the sets $C_1$ (in the left) and $C_2$ (in the right). The point $x\in (\xi^1_k,\xi^1_{k+1})$ is displayed in gray while the discontinuity point is displayed in red.}
    \label{xleft}
\end{figure}
\begin{figure}
    \centering
  \includegraphics[width=0.8\linewidth]{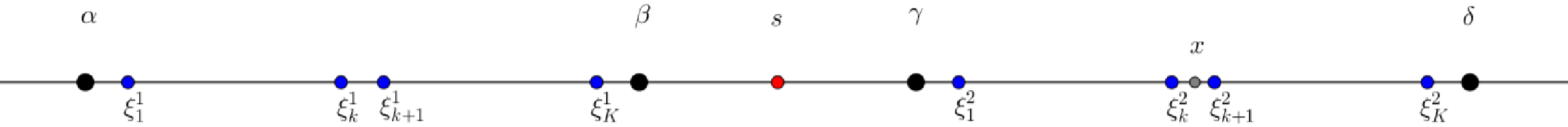}
    \caption{Plot of $K$ equispaced points (in blue) for each of the sets $C_1$ (in the left) and $C_2$ (in the right). The point $x\in (\xi^2_k,\xi^2_{k+1})$ is displayed in gray while the discontinuity point is displayed in red.}
    \label{xright}
\end{figure}

\begin{lemma}\label{lemmaboundBmu}
Let $\alpha,\beta,\gamma,\delta\in \mathbb{R}$ be such that $\alpha<\beta\leq\gamma<\delta$ with
\begin{equation}\label{abc}
  \beta-\alpha=\delta-\gamma=h^{\max,d}_{X_n}
\end{equation}
 and let  
\begin{equation*}
  C_{1}=\left\{\xi^{1}_{\kappa}=\alpha+\frac{\beta-\alpha}{K+1}\kappa \, :\, \kappa=1,\dots,K \right\} \subset (\alpha,\beta), 
\end{equation*}  
\begin{equation*}
  C_{2}=\left\{\xi^{2}_{\kappa}=\gamma+\frac{\delta-\gamma}{K+1}\kappa \, :\, \kappa=1,\dots,K \right\} \subset (\gamma,\delta).
\end{equation*}
Then, for any $x\in \left(\xi^1_{k},\xi^1_{k+1}\right)$, $k=1,\dots,K-1$, we have 
\begin{equation} \label{Bmu2boundgen}
    B_{\mu,2}(x)\leq \left(\frac{k!(K-k)!}{\prod\limits_{\kappa=1}^K \left(K-k+\kappa\right)+(K+1)^K\left(\frac{\gamma-\beta}{h^{\max,d}_{X_n}}\right)^K}\right)^{\mu}=F_{\mu}(\gamma-\beta,h^{\max,d}_{X_n},k,K).
\end{equation}
Similarly, for any $x\in \left(\xi^2_{k},\xi^2_{k+1}\right)$, $k=1,\dots,K-1$, we have 
\begin{equation} \label{Bmu1boundgen}
    B_{\mu,1}(x) \leq \left(\frac{k!(K-k)!}{\prod\limits_{\kappa=1}^K \left(k+\kappa\right)+(K+1)^K\left(\frac{\gamma-\beta}{h^{\max,d}_{X_n}}\right)^K}\right)^{\mu}=F_{\mu}(\gamma-\beta,h^{\max,d}_{X_n},K-k,K). 
\end{equation} 
\end{lemma}

\begin{proof}
 Let $x\in \left(\xi^1_{k},\xi^1_{k+1}\right)$, $k=1,\dots,K-1$. Then we have
    \begin{eqnarray*}
B_{\mu,2}(x)&=&\frac{\prod\limits_{\kappa=1}^{K} \left\lvert x-\xi^2_{{\kappa}} \right\rvert^{-\mu}}{\prod\limits_{\kappa=1}^{K} \left\lvert x-\xi^1_{{\kappa}} \right\rvert^{-\mu}+\prod\limits_{\kappa=1}^{K} \left\lvert x-\xi^2_{{\kappa}} \right\rvert^{-\mu}} \leq \frac{\prod\limits_{\kappa=1}^{K} \left\lvert x-\xi^1_{{\kappa}} \right\rvert^{\mu}}{\prod\limits_{\kappa=1}^{K} \left\lvert x-\xi^2_{{\kappa}} \right\rvert^{\mu}} \\
&\leq& \frac{\left\lvert x-\xi^1_{1} \right\rvert^{\mu} \cdots \left\lvert x-\xi^1_{k} \right\rvert^{\mu} \cdot \left\lvert x-\xi^1_{k+1} \right\rvert^{\mu} \cdots \left\lvert x-\xi^1_{K} \right\rvert^{\mu}}{\left\lvert \xi^1_{k+1}-\xi^2_{1} \right\rvert^{\mu} \cdot \left\lvert \xi^1_{k+1}-\xi^2_{2} \right\rvert^{\mu} \cdots \left\lvert \xi^1_{k+1}-\xi^2_{K-1} \right\rvert^{\mu} \cdot \left\lvert \xi^1_{k+1}-\xi^2_{K} \right\rvert^{\mu}} \\
&\leq&  \left(\frac{(\beta-\alpha)^{K}\frac{1+k-1}{K+1}\cdots \frac{1+0}{K+1}\cdot\frac{1+0}{K+1}\cdots \frac{1+K-k-1}{K+1}}{\left(\frac{(\beta-\alpha)(K-k)+(\delta-\gamma)}{K+1}+(\gamma-\beta)\right)\cdots\left(\frac{(\beta-\alpha)(K-k)+(\delta-\gamma)K}{K+1}+(\gamma-\beta)\right)}\right)^{\mu} \\
&\leq&  \left(\frac{(\beta-\alpha)^{K}k!(K-k)!}{\prod\limits_{\kappa=1}^K \left(\left(\beta-\alpha\right)\left(K-k\right)+\kappa\left(\delta-\gamma\right)\right)+(K+1)^K(\gamma-\beta)^K}\right)^{\mu}\\
&\leq&  \left(\frac{k!(K-k)!}{\prod\limits_{\kappa=1}^K \left(K-k+\kappa\right)+(K+1)^K\left(\frac{\gamma-\beta}{\beta-\alpha}\right)^K}\right)^{\mu}\\
&\leq&  \left(\frac{k!(K-k)!}{\prod\limits_{\kappa=1}^K \left(K-k+\kappa\right)+(K+1)^K\left(\frac{\gamma-\beta}{h^{\max,d}_{X_n}}\right)^K}\right)^{\mu},
\end{eqnarray*}
where in the last inequality we used the hypothesis~\eqref{abc}.
Similarly, if $x\in \left(\xi^2_{k},\xi^2_{k+1}\right)$, $k=1,\dots,K-1$, we have 
\begin{eqnarray*}
B_{\mu,1}(x)&=&\frac{\prod\limits_{\kappa=1}^{K} \left\lvert x-\xi^1_{{\kappa}} \right\rvert^{-\mu}}{\prod\limits_{\kappa=1}^{K} \left\lvert x-\xi^1_{{\kappa}} \right\rvert^{-\mu}+\prod\limits_{\kappa=1}^{K} \left\lvert x-\xi^2_{{\kappa}} \right\rvert^{-\mu}} \leq \frac{\prod\limits_{\kappa=1}^{K} \left\lvert x-\xi^2_{{\kappa}} \right\rvert^{\mu}}{\prod\limits_{\kappa=1}^{K} \left\lvert x-\xi^1_{{\kappa}} \right\rvert^{\mu}} \\
&\leq& \frac{\left\lvert x-\xi^2_{1} \right\rvert^{\mu} \cdots \left\lvert x-\xi^2_{k} \right\rvert^{\mu} \cdot \left\lvert x-\xi^2_{k+1} \right\rvert^{\mu} \cdots \left\lvert x-\xi^2_{K} \right\rvert^{\mu}}{\left\lvert \xi^2_{k}-\xi^1_{1} \right\rvert^{\mu} \cdot \left\lvert \xi^2_{k}-\xi^1_{2} \right\rvert^{\mu} \cdots \left\lvert \xi^2_{k}-\xi^1_{K-1} \right\rvert^{\mu} \cdot \left\lvert \xi^2_{k}-\xi^1_{K} \right\rvert^{\mu}} \\
&\leq&  \left(\frac{(d-c)^{K}\frac{1+k-1}{K+1}\cdots \frac{1+0}{K+1}\cdot\frac{1+0}{K+1}\cdots \frac{1+K-k-1}{K+1}}{\left(\frac{(\delta-\gamma)k+(\beta-\alpha)K}{K+1}+(\gamma-\beta)\right)\cdots\left(\frac{(\delta-\gamma)k+(\beta-\alpha)}{K+1}+(\gamma-\beta)\right)}\right)^{\mu} \\
&\leq&  \left(\frac{(\delta-\gamma)^{K}k!(K-k)!}{\prod\limits_{\kappa=1}^K \left(\left(\delta-\gamma\right)k+\kappa\left(\beta-\alpha\right)\right)+(K+1)^K(\gamma-\beta)^K}\right)^{\mu}\\
&\leq&  \left(\frac{k!(K-k)!}{\prod\limits_{\kappa=1}^K \left(k+\kappa\right)+(K+1)^K\left(\frac{\gamma-\beta}{\delta-\gamma}\right)^K}\right)^{\mu}\\
&\leq&  \left(\frac{k!(K-k)!}{\prod\limits_{\kappa=1}^K \left(k+\kappa\right)+(K+1)^K\left(\frac{\gamma-\beta}{h^{\max,d}_{X_n}}\right)^K}\right)^{\mu}.    
\end{eqnarray*}
\end{proof}
\begin{proposition} \label{proposition}
    For any $K\in\mathbb{N}$ and $k=1,\dots,K-1$, we have 
    \begin{equation}
         F_{\mu}(\gamma-\beta,h^{\max,d}_{X_n},k,K)\leq F_{\mu}(\gamma-\beta,h^{\max,d}_{X_n},K-1,K).
    \end{equation}
\end{proposition}
\begin{proof}
  We notice that, for any $K \in \mathbb{N}$, we have       
    \begin{equation*}
        k!(K-k)! \leq (K-1)!  \quad k=1,\dots,K-1.
    \end{equation*}
    In fact
    \begin{eqnarray*}
        \frac{k!(K-k)!}{(K-1)!} = \dfrac{\prod\limits_{j=0}^{k-2}(k-j)}{\prod\limits_{j=0}^{k-2}(K-1-j)} = \prod\limits_{j=0}^{k-2}\dfrac{k-j}{K-1-j} \leq 1.
    \end{eqnarray*} 
    Therefore
    \begin{eqnarray*}
        F_{\mu}(\gamma-\beta,h^{\max,d}_{X_n},k,K)&=& \left(\frac{k!(K-k)!}{\prod\limits_{\kappa=1}^K \left(K-k+\kappa\right)+(K+1)^K\left(\frac{\gamma-\beta}{h^{\max,d}_{X_n}}\right)^K}\right)^{\mu} \\
        &\leq& \left(\frac{(K-1)!}{\prod\limits_{\kappa=1}^K \left(1+\kappa\right)+(K+1)^K\left(\frac{\gamma-\beta}{h^{\max,d}_{X_n}}\right)^K}\right)^{\mu}\\
        &=& F_{\mu}(\gamma-\beta,h^{\max,d}_{X_n},K-1,K).
    \end{eqnarray*}
\end{proof}

\begin{theorem}\label{theoremerrorboundp}
    Let $ f \in C^{n_{\max}+1}(\Xi) $. Then
    \begin{eqnarray}\label{errboundp}
        \left\lvert f(x)-Q_{\mu}[f](x) \right\rvert  &\leq&  \max_{\substack{\iota=1,\dots,M \\ x\in U_{\iota}}}  \frac{\max\limits_{y \in \displaystyle{\cap_{\substack{\iota=1,\dots,M \\ x\in U_{\iota}}} U_{\iota}}} \left\lvert f^{\left(\mathrm{card}\left(N_{\iota}\right)+1\right)} (y) \right\rvert }{\left(\mathrm{card}\left(N_{\iota}\right)+1\right)!}   \left\lvert \omega_{\iota}(x)\right\rvert \\ \nonumber &+&  M\left
    (\left\lvert f({x})\right\rvert + \max_{\substack{\iota=1,\dots,M \\ x\notin U_{\iota}}}\left\lvert p_{\iota}(f,x)\right\rvert \right)  \max_{\substack{\iota=1,\dots,M \\ x\notin U_{\iota}}}F_{\mu}\left(a_{\iota+1}-b_{\iota},h^{\max,d}_{X_n}, K-1,K\right), 
    \end{eqnarray}
    where the second part of~\eqref{errboundp} can be made arbitrarily small by increasing $K$.
\end{theorem}
\begin{proof}
Let $x\in \Xi$, in particular, there exists $\Bar{\ell}\in \{0,\dots,m\}$ such that $x \in \left[\xi^{\Bar{\ell}}_1,\xi^{n_{\Bar{\ell}}}_K\right]$. Using~\eqref{Qoperator} and leveraging both the partition of unity~\eqref{partofuni} and the non-negative properties of the multinode Shepard functions~\eqref{nonnegprop}, we obtain

	\begin{eqnarray}	
	e({x})&=&\left\lvert f({x})- Q_{\mu}[f](x)\right\rvert \notag \nonumber \\
	&=&\left\lvert \sum_{\iota=1}^M	B_{\mu ,\iota}\left( {x}\right)f({x})-\sum_{\iota=1}^M	B_{\mu,\iota}\left(x\right)p_{\iota}(f,x)\right\rvert \notag  \nonumber \\
	&\leq&\sum_{\iota=1}^M\left\lvert f({x})-p_{\iota}(f,x)\right\rvert B_{\mu ,\iota}\left(x\right)  \label{sumdifffpol} \\
    &\leq &  \sum_{\substack{\iota=1,\dots,M \\ x\in U_{\iota}}}\left\lvert f({x})-p_{\iota}(f,x)\right\rvert B_{\mu ,\iota}\left(x\right) +\sum_{\substack{\iota=1,\dots,M \\ x\notin U_{\iota}}} \left
    (\left\lvert f({x})\right\rvert + \left\lvert p_{\iota}(f,x)\right\rvert \right)
    B_{\mu ,\iota}\left(x\right) \nonumber \\
    &\le&   \max_{\substack{\iota=1,\dots,M \\ x\in U_{\iota}}}  \frac{\max\limits_{y \in \displaystyle{\cap_{\substack{\iota=1,\dots,M \\ x\in U_{\iota}}} U_{\iota}}} \left\lvert f^{\left(\mathrm{card}\left(N_{\iota}\right)+1\right)} (y) \right\rvert }{\left(\mathrm{card}\left(N_{\iota}\right)+1\right)!}   \left\lvert \omega_{\iota}(x)\right\rvert \nonumber \\ &+&  \left
    (\left\lvert f({x})\right\rvert + \max_{\substack{\iota=1,\dots,M \\ x\notin U_{\iota}}}\left\lvert p_{\iota}(f,x)\right\rvert \right) \sum_{\substack{\iota=1,\dots,M \\ x\notin U_{\iota}}}  B_{\mu ,\iota}(x) \nonumber \\ 
        &\le&   \max_{\substack{\iota=1,\dots,M \\ x\in U_{\iota}}}  \frac{\max\limits_{y \in \displaystyle{\cap_{\substack{\iota=1,\dots,M \\ x\in U_{\iota}}} U_{\iota}}} \left\lvert f^{\left(\mathrm{card}\left(N_{\iota}\right)+1\right)} (y) \right\rvert }{\left(\mathrm{card}\left(N_{\iota}\right)+1\right)!}   \left\lvert \omega_{\iota}(x)\right\rvert \nonumber \\ &+& \left
    (\left\lvert f({x})\right\rvert + \max_{\substack{\iota=1,\dots,M \\ x\notin U_{\iota}}}\left\lvert p_{\iota}(f,x)\right\rvert \right) \sum_{\substack{\iota=1,\dots,M \\ x\notin U_{\iota}}}  F_{\mu}\left(a_{\iota+1}-b_{\iota},h^{\max,d}_{X_n}, K-1,K\right) \nonumber \\
    &\leq& \max_{\substack{\iota=1,\dots,M \\ x\in U_{\iota}}}  \frac{\max\limits_{y \in \displaystyle{\cap_{\substack{\iota=1,\dots,M \\ x\in U_{\iota}}} U_{\iota}}} \left\lvert f^{\left(\mathrm{card}\left(N_{\iota}\right)+1\right)} (y) \right\rvert }{\left(\mathrm{card}\left(N_{\iota}\right)+1\right)!}   \left\lvert \omega_{\iota}(x)\right\rvert \nonumber \\ &+&  M\left
    (\left\lvert f({x})\right\rvert + \max_{\substack{\iota=1,\dots,M \\ x\notin U_{\iota}}}\left\lvert p_{\iota}(f,x)\right\rvert \right)  \max_{\substack{\iota=1,\dots,M \\ x\notin U_{\iota}}}F_{\mu}\left(a_{\iota+1}-b_{\iota},h^{\max,d}_{X_n}, K-1,K\right), \nonumber
\end{eqnarray}
where $\omega_{\iota}(x)$ is the nodal polynomial of degree $\mathrm{card}\left(N_{\iota}\right)$ which vanishes on the nodes of $N_{\iota}$~\cite{Gautschi:2011:NA}.
\end{proof}

In the following, we determine a bound in the infinity norm of the approximation error provided by the quasi-interpolation operator~\eqref{Qoperator}.
For this purpose, further notations are needed. Let 
\begin{equation*}
   n_{\max}=\max\limits_{\iota=1,\dots,M} \mathrm{card}\left(N_{\iota}\right),
\end{equation*}
\begin{equation*}
    \left\lVert f - Q_{\mu}[f]\right\rVert_{\infty,\Xi}=\max\limits
    _{x \in \Xi} \left\lvert f(x)-Q_{\mu}[f](x) \right\rvert,
\end{equation*}
and
\begin{equation*}
    h^S_{\min}=\min\limits_{i=1,\dots,m} \left\lvert x_{\alpha_i+1}-x_{\alpha_i} \right\rvert.
\end{equation*}
By fixing $y \in [a,b]$, we denote also by
      \begin{equation*}
         \mathcal{R}_{h^{\max,d}_{X_n}}(y)=\left\{{x}\in \mathbb{R}\, : \, y-h^{\max,d}_{X_n}<x\le y+h^{\max,d}_{X_n}
         \right\}
      \end{equation*}
the half-open interval of radius $h^{\max,d}_{X_n}$ centered at $y$ and by
\begin{equation}\label{valM}
    \mathcal{M}_{h^{\max,d}_{X_n}}= \sup_{y\in [a,b]}\mathrm{card}\left(\left\{C_{\iota}\, :\,  C_{\iota} \cap \mathcal{R}_{h^{\max,d}_{X_n}}({{y}}) \neq \emptyset, \, \iota=1,\dots,M\right\}\right),
\end{equation}
the maximum number of the sets $C_{\iota}$ intersecting $\mathcal{R}_{h^{\max,d}_{X_n}}(y)$.  
We notice that if $y=\frac{a_{\iota}+b_{\iota}}{2}$ then, by construction, $\mathcal{M}_{h^{\max,d}_{X_n}}\leq 2$, while in all other cases $\mathcal{M}_{h^{\max,d}_{X_n}}\leq 4$. Finally, we set 
\begin{equation}\label{Meffe}
    \mathcal{M}_{f}=\max_{\iota=1,\dots,M} \max_{\ell=0,\dots,m} \max_{y\in I_{\ell}}\left\lvert f^{\left(\mathrm{card}\left(N_{\iota}\right)+1\right)} (y) \right\rvert.
 \end{equation}
  
\begin{theorem} \label{errorbound}
   Let $ f \in C^{n_{\max}+1}(\Xi) $ and assume that $\mu>\dfrac{n_{\max}+2}{K}$. Then
    \begin{equation}
        \left\lVert f-Q_{\mu}[f] \right\rVert_{\infty,\Xi}  \leq \mathcal{E}_1+\mathcal{E}_2, 
    \end{equation}
    where
        \begin{equation}
        \mathcal{E}_1=\mathcal{C}_1 \max\limits_{\iota=1,\dots,M}\frac{ (2h^{\max,d}_{X_n})^{\mathrm{card}\left(N_{\iota}\right)+1}}{\left(\mathrm{card}\left(N_{\iota}\right)+1\right)!}
    \end{equation}
   tends to zero when $h_{X_n}$ tends to zero and 
    \begin{equation}
     \mathcal{E}_2 =  \mathcal{C}_2 F_{\mu}\left(h^S_{\min},h^{\max,d}_{X_n},K-1,K\right) 
    \end{equation}
    can be made arbitrarily small by increasing $K$. 
\end{theorem}
\begin{remark}
    The constant 
    \begin{equation*}
        \mathcal{C}_1=\mathcal{M}_{f} \mathcal{M}_{h^{\max,d}_{X_n}} \left( 1+2\cdot2^{n_{\max}+1}+2\sum\limits_{\theta=2}^{\infty}   	\frac{(\theta+1)^{n_{\max}+1}}{(\theta-\frac{1}{2})^{K\mu}}  \right)
    \end{equation*}
    depends only on $f$ (and its derivatives up to order $n_{\max}+1$), $K$ and $\mu$.
    The constant 
    \begin{equation*}
         \mathcal{C}_2=2M \left\lVert f \right\rVert_{\infty} \left(1 + \max_{\ell=0,\dots,m}\max_{\substack{\iota=1,\dots,M \\ U_{\iota}\cap I_{\ell} = \emptyset}} \max\limits_{y\in[a,b]}\lambda_{\mathrm{card}\left(N_{\iota}\right)}(y) \right)
    \end{equation*}
    depends only on $f$, $M$ and the distribution of the interpolation nodes in the interval $[a,b]$.
\end{remark}
\begin{proof}
Let $x\in \Xi$, in particular, there exists $\Bar{\ell}\in \{0,\dots,m\}$ such that $x \in \left[\xi^{\Bar{\ell}}_1,\xi^{n_{\Bar{\ell}}}_K\right]$. We consider the minimal covering $\{A_{\theta}\}_{\theta=0}^{\Theta}$  
      of the interval $\left[\xi^{\Bar{\ell}}_1,\xi^{n_{\Bar{\ell}}}_K\right]$ by the following subsets
      \begin{equation*}
    A_0=\mathcal{R}_{h^{\max,d}_{X_n}}(x)
\end{equation*}
and
      \begin{equation*}
A_{\theta}=\mathcal{R}_{h^{\max,d}_{X_n}}({x}- 2 h^{\max,d}_{X_n}\theta)\cup \mathcal{R}_{h^{\max,d}_{X_n}}({x}+ 2 h^{\max,d}_{X_n} \theta), \quad \theta=1,\dots, \Theta.
\end{equation*} 
Clearly
\begin{equation*}
    \left[\xi^{\Bar{\ell}}_1,\xi^{n_{\Bar{\ell}}}_K\right] \subset \bigcup_{\theta=0}^{\Theta}A_{\theta}.
\end{equation*}
Since $A_{\theta}$ is composed of two congruent copies of $\mathcal{R}_{h^{\max,d}_{X_n}}({x})$ if $\theta=1,2,\dots,\Theta$, the number of subsets $C_{\Bar{\ell},j}$, $j=0,\dots,n_{\Bar{\ell}}$, with at least one vertex in $A_{\theta}$, $\theta=1,2,\dots,\Theta$, is bounded by $2 \mathcal{M}_{h^{\max,d}_{X_n}}$. 

Using~\eqref{sumdifffpol}, it follows 
\begin{eqnarray}\label{err1}
    e({x}) &\leq& \sum_{\iota=1}^M\left\lvert f({x})-p_{\iota}(f,x)\right\rvert B_{\mu ,\iota}\left(x\right) \notag \\ &=&  \sum_{\substack{\iota=1,\dots,M \\ U_{\iota}\cap I_{\Bar{\ell}} \neq \emptyset}}\left\lvert f({x})-p_{\iota}(f,x)\right\rvert B_{\mu ,\iota}\left(x\right) + 
      \sum_{\substack{\iota=1,\dots,M \\ U_{\iota}\cap I_{\Bar{\ell}} = \emptyset}}\left\lvert f({x})-p_{\iota}(f,x)\right\rvert B_{\mu ,\iota}\left(x\right) \\ \label{e1,2}
      &=:& e_1(x)+e_2(x).
\end{eqnarray}
We start by bounding the first part. 
We denote by $V_0$ the set of $C_{\iota}$ with at least one node in $A_0$. Analogously, for $\theta= 1,\dots,\Theta$, we further denote by $V_{\theta}$ the set of $C_{\iota}$ with at least one node in $A_{\theta}$ and no nodes in $A_{\theta-1}$. By construction, we have 
\begin{equation*}
    \bigcup_{\theta=0}^\Theta V_{\theta}=\bigcup_{\iota=1}^M C_{\iota} \quad \text{ and } \quad \bigcap_{\theta=0}^\Theta V_{\theta}=\emptyset.
\end{equation*} 

Let assume that
\begin{equation*}
x\in\bigcup\limits_{\substack{\iota=1,\dots,M \\ U_{\iota} \cap I_{\Bar{\ell}} \neq \emptyset}} C_{\iota} 
\end{equation*} 
then, by the Kronecker delta property of the multinode Shepard functions~\cite{DellAccio:2018:ROF}, from equation~\eqref{err1} and since $f \in C^{n_{\max}+1}(\Xi)$, we get
\begin{equation}\label{err1B1}
    e_1(x)\leq \left\lvert f({x})-p_{\iota} (f,x)\right\rvert\leq   \frac{\left\lvert\omega_{\iota}(x)\right\rvert}{\left(\mathrm{card}\left(N_{\iota}\right)+1\right)!}  \max\limits_{y \in I_{\bar{\ell}}} \left\lvert f^{\left(\mathrm{card}\left(N_{\iota}\right)+1\right)} (y) \right\rvert.
\end{equation}

Now, we assume that
\begin{equation*}
    {x}\notin \bigcup\limits_{\substack{\iota=1,\dots,M \\ U_{\iota} \cap I_{\bar{\ell}} \neq \emptyset}} C_{\iota}
\end{equation*}
 and let $\iota_{\min}\in \{1,\dots, M\}$ such that 
\begin{equation*}
\prod\limits_{\kappa=1}^{K}\left\lvert {x}-{\xi}^{\iota_{\min}}_{\kappa}\right\rvert= \underset{\iota=1,\dots,M}{\min} 	\prod\limits_{\kappa=1}^{K}\left\vert {x}-{\xi}^{\iota}_{\kappa}\right\vert.   
\end{equation*}
Observe that, if $C_{\iota} \in V_0$, we have 
\begin{equation}\label{ineqinR0}
\prod\limits_{\kappa=1}^{K}\left\lvert{x}-{\xi}^{\iota}_{\kappa}\right\rvert\leq  \left(2h^{\max,d}_{X_n}\right)^{K},    
\end{equation}
while if $C_{\iota} \in V_{\theta}$, $\theta=1,\dots, \Theta$, then
\begin{equation}\label{ineqinRthe}
((2\theta-1) h^{\max,d}_{X_n})^{K} \leq \prod\limits_{\kappa=1}^{K}\left\vert {x}-{\xi}^{\iota}_{\kappa}\right\vert
 \leq \left((2\theta+2) h^{\max,d}_{X_n} \right)^{K}.    
\end{equation}
Finally, we get
\begin{equation*}
		B_{\mu ,\iota}(x)=\dfrac{\prod\limits_{\kappa=1}^{K}\left\vert 
{x}-{\xi}^{\iota}_{\kappa}\right\vert ^{-\mu }}{\sum\limits_{\tau=1}^{M}%
\prod\limits_{\kappa=1}^{K}\left\vert {x}-{\xi}^{\tau}_{\kappa}\right\vert
^{-\mu}}\leq \min \left\{  1,\prod\limits_{\kappa=1}^{K}\frac{\left\lvert{x}-{\xi}^{\iota}_{\kappa}\right\rvert^{-\mu}}{\left\lvert {x}-{\xi}^{\iota_{\min}}_{\kappa}\right\rvert^{-\mu}}\right\} 
\leq \left\{ \begin{array}{ll}
			1, & C_{\iota} \in V_{\theta}, \, \theta=0,1,\\
			\frac{1}{(\theta-\frac{1}{2})^{K\mu}},& C_{\iota} \in V_{\theta}, \theta \geq 2.
		\end{array}\right. 
\end{equation*} 
We underline the fact that, since $C_{\iota} \subset U_{\iota}$ the inequalities~\eqref{ineqinR0} and~\eqref{ineqinRthe} hold also for any chosen set $$N_{\iota} \subset U_{\iota},$$ for $ \iota=1,\dots,M$. 
By~\eqref{err1},~\eqref{err1B1} and
after easy computations, it results
\begin{eqnarray*}
    	e_1(x) & \leq& \mathcal{M}_{f} \left( \sum\limits_{C_{\iota} \in V_{0}\cup V_{1}}\frac{\left\lvert \omega_{\iota}(x)\right\rvert}{\left(\mathrm{card}\left(N_{\iota}\right)+1\right)!}+\sum\limits_{\theta=2}^{\Theta} \sum\limits_{C_{\iota} \in V_{\theta}}\frac{\left\lvert \omega_{\iota}(x)\right\rvert}{\left(\mathrm{card}\left(N_{\iota}\right)+1\right)!} 	\frac{1}{(\theta-\frac{1}{2})^{K\mu}}  \right) \\
     & \leq& \mathcal{M}_{f} \left( \sum\limits_{C_{\iota} \in V_{0}}\frac{ (2h^{\max,d}_{X_n})^{\mathrm{card}\left(N_{\iota}\right)+1}}{\left(\mathrm{card}\left(N_{\iota}\right)+1\right)!}+\sum\limits_{C_{\iota} \in V_{1}}\frac{ (4h^{\max,d}_{X_n})^{\mathrm{card}\left(N_{\iota}\right)+1}}{\left(\mathrm{card}\left(N_{\iota}\right)+1\right)!}
     \right.
     \\
     &&+ \left.\sum\limits_{\theta=2}^{\Theta} \sum\limits_{C_{\iota} \in V_{\theta}}\frac{\left((2\theta+2) h^{\max,d}_{X_n} \right)^{\mathrm{card}\left(N_{\iota}\right)+1}}{\left(\mathrm{card}\left(N_{\iota}\right)+1\right)!} 	\frac{1}{(\theta-\frac{1}{2})^{K\mu}} \right) \\
      & \leq& \mathcal{M}_{f} \mathcal{M}_{h^{\max,d}_{X_n}} \max\limits_{\iota=1,\dots,M}\frac{ (2h^{\max,d}_{X_n})^{\mathrm{card}\left(N_{\iota}\right)+1}}{\left(\mathrm{card}\left(N_{\iota}\right)+1\right)!} \left( 1+2\cdot2^{n_{\max}+1}+\sum\limits_{\theta=2}^{\Theta} 2 (\theta+1)^{n_{\max}+1} 	\frac{1}{(\theta-\frac{1}{2})^{K\mu}}  \right) \\
      &\leq&\mathcal{M}_{f} \mathcal{M}_{h^{\max,d}_{X_n}} \max\limits_{\iota=1,\dots,M}\frac{ (2h^{\max,d}_{X_n})^{\mathrm{card}\left(N_{\iota}\right)+1}}{\left(\mathrm{card}\left(N_{\iota}\right)+1\right)!} \left( 1+2\cdot2^{n_{\max}+1}+2\sum\limits_{\theta=2}^{\Theta}   	\frac{(\theta+1)^{n_{\max}+1}}{(\theta-\frac{1}{2})^{K\mu}}  \right) \\
       &\leq&\mathcal{M}_{f} \mathcal{M}_{h^{\max,d}_{X_n}} \max\limits_{\iota=1,\dots,M}\frac{ (2h^{\max,d}_{X_n})^{\mathrm{card}\left(N_{\iota}\right)+1}}{\left(\mathrm{card}\left(N_{\iota}\right)+1\right)!} \left( 1+2\cdot2^{n_{\max}+1}+2\sum\limits_{\theta=2}^{\infty}   	\frac{(\theta+1)^{n_{\max}+1}}{(\theta-\frac{1}{2})^{K\mu}}  \right).
     \end{eqnarray*}
     The series $\sum\limits_{\theta=2}^{\infty} \frac{(\theta+1)^{n_{\max}+1}}{(\theta-\frac{1}{2})^{K\mu}}$ converges for $\mu>\frac{n_{\max}+2}{K}$.
It remains to bound $e_2(x)$ defined in~\eqref{e1,2}. Since $x \in \left[\xi^{\Bar{\ell}}_1,\xi^{n_{\Bar{\ell}}}_K\right]$ we can assume that there exist $i=\Bar{\ell},\dots,n_{\Bar{\ell}}$ and $k=1,\dots,K-1$ such that $x \in \left[\xi^i_k, \xi^i_{k+1}\right]$.
 In order to bound the multinode Shepard functions $B_{\mu,\iota}$, we use two times Lemma~\ref{lemmaboundBmu} and Proposition~\ref{proposition}. More precisely, we replace
 \begin{eqnarray*}
 \alpha\rightarrow a_{\iota-1}, \quad \beta \rightarrow b_{\iota-1}, \quad \gamma \rightarrow a_{\iota}, \quad \delta \rightarrow b_{\iota},
\end{eqnarray*}
for $\iota=1,\dots,M$ such that $U_{\iota}\cap I_{\Bar{\ell}} =\emptyset$ and $\xi^{\iota}_K < \xi^i_k$ and we replace  
\begin{eqnarray*}
 \alpha \rightarrow a_{\iota}, \quad \beta \rightarrow b_{\iota}, \quad \gamma \rightarrow a_{\iota+1}, \quad \delta \rightarrow b_{\iota+1},
\end{eqnarray*}
for $\iota=1,\dots,M$ such that $U_{\iota}\cap I_{\Bar{\ell}} =\emptyset$ and $\xi^{\iota}_1 > \xi^i_{k+1}$. By using the triangular inequality, we have 

\begin{eqnarray*}
    e_2(x) &=& \sum_{\substack{\iota=1,\dots,M \\ U_{\iota}\cap I_{\Bar{\ell}} = \emptyset}}\left\lvert f({x})-p_{\iota}(f,x)\right\rvert B_{\mu ,\iota}\left(x\right) \\ 
    &\leq& \sum_{\substack{\iota=1,\dots,M \\ U_{\iota}\cap I_{\Bar{\ell}} = \emptyset}} \left
    (\left\lvert f({x})\right\rvert + \left\lvert p_{\iota}(f,x)\right\rvert \right)
    B_{\mu ,\iota}\left(x\right) \\
    &\leq& \left
    (\left\lvert f({x})\right\rvert + \max_{\substack{\iota=1,\dots,M \\ U_{\iota}\cap I_{\Bar{\ell}} = \emptyset}}\left\lvert p_{\iota}(f,x)\right\rvert \right) \sum_{\substack{\iota=1,\dots,M \\ U_{\iota}\cap I_{\Bar{\ell}} = \emptyset}}  B_{\mu ,\iota}(x) \\
    &\leq& \left
    (\left\lVert f \right\rVert_{\infty} + \max_{\substack{\iota=1,\dots,M \\ U_{\iota}\cap I_{\Bar{\ell}} = \emptyset}}\left\lVert p_{\iota}(f,\cdot)\right\rVert_{\infty} \right) \times\\
    && \left( c_1 \cdot F_{\mu}\left(a_{\iota}-b_{\iota-1},h^{\max,d}_{X_n},K-k,K\right)  + c_2 \cdot F_{\mu}\left(a_{\iota+1}-b_{\iota},h^{\max,d}_{X_n},k,K\right)\right) \\
    &\leq& 
    M \left\lVert f \right\rVert_{\infty} \left(1 + \max_{\substack{\iota=1,\dots,M \\ U_{\iota}\cap I_{\Bar{\ell}} = \emptyset}} \max\limits_{y\in[a,b]}\lambda_{\mathrm{card}\left(N_{\iota}\right)}(y) \right) \times 
    \\
    &&
\left(F_{\mu}\left(h^S_{\min},h^{\max,d}_{X_n},1,K\right)+F_{\mu}\left(h^S_{\min},h^{\max,d}_{X_n},K-1,K\right)\right)\\
    &\leq& 
    2M \left\lVert f \right\rVert_{\infty} \left(1 + \max_{\ell=0,\dots,m}\max_{\substack{\iota=1,\dots,M \\ U_{\iota}\cap I_{\ell} = \emptyset}} \max\limits_{y\in[a,b]}\lambda_{\mathrm{card}\left(N_{\iota}\right)}(y) \right)  F_{\mu}\left(h^S_{\min},h^{\max,d}_{X_n},K-1,K\right), 
\end{eqnarray*}
where 
$$c_1=\mathrm{card}\left(\left\{\iota=1,\dots,M : U_{\iota}\cap I_{\Bar{\ell}} =\emptyset \wedge \xi^{\iota}_K < \xi^i_k  \right\}\right),$$

$$c_2=\mathrm{card}\left(\left\{\iota=1,\dots,M : U_{\iota}\cap I_{\Bar{\ell}} =\emptyset \wedge \xi^{\iota}_1 > \xi^i_{k+1}  \right\}\right)$$
and $\lambda_{\mathrm{card}\left(N_{\iota}\right)}(y)$ is the Lebesgue function for Lagrange interpolation in $N_{\iota}$ and $\max\limits_{y\in[a,b]}\lambda_{\mathrm{card}\left(N_{\iota}\right)}(y)$ is the relative Lebesgue constant. 
\end{proof}
\begin{remark}\label{remarknoise}
    We notice that, for any $\iota=1,\dots,M$,  the local polynomial interpolant on $N_{\iota}$, used in the definition of the quasi-interpolation operator~\eqref{Qoperator}, can be substituted with a least squares polynomial on the same nodes. This alternative strategy allows us to better reconstruct functions whose samplings are affected by noise.
\end{remark}

\section{Numerical experiments}
\label{sec4}
In this section, we perform a series of numerical experiments to confirm the order of approximation theoretically proved in the previous section and the absence of Gibbs-like oscillations. 
Moreover, we show how the approximations worsen as we get closer and cross the singularity  $s_i \in \left(x_{\alpha_i}, x_{\alpha_{i}+1}\right)$.
Finally, we prove that the algorithm is robust to the presence of perturbations in the data by presenting experiments with perturbed data functions using additive white Gaussian noise.   

In the following numerical experiments, we consider the test functions $f_1$, $f_2$~\cite{Amat:2022:ACO}, $f_3$~\cite{Arandiga:2024:EAW} and $f_4$~\cite{Costarelli:2017:ADS}, scaled in $[-1,1]$ and defined as follows 
\begin{equation}
    f_1(x)=\begin{cases}
        \sin\left(\frac{17}{8} \pi x\right) & x \leq 0, \\
        \frac{1}{2}\sin\left(\frac{17}{8} \pi x\right) + 10 & x>0,     
    \end{cases}
    \quad
    f_2(x)=\begin{cases}
        \frac{1}{2} x^5-x^2 & x \leq 0, \\
        x^6-x^4+x^2-2 & x>0,     
    \end{cases} 
    \end{equation}
    \begin{equation}
    f_3(x)=\begin{cases}
        e^{\frac{1}{2}(x+1)} &  x \leq 0, \\
        1+e^{\frac{1}{4}(x+1)^2} & x>0,  
    \end{cases}
    \quad
    f_4(x)=\begin{cases}
        \frac{5}{\left(\frac{x}{4}\right)^2+1} & \lvert x \rvert \geq \frac{1}{2}, \\
        \frac{3}{2} & -\frac{1}{2} < x < 0, \\
        \frac{1}{4} & 0 \leq x < \frac{1}{2}.
    \end{cases}
\end{equation}
\subsection{Numerical experiment $1$}\label{subsection1}
In the first class of numerical experiments, we assume to know the evaluations of the functions $f_i$, $i=1,2,3,4$, on the grid of $n+1=1025$ equispaced nodes in $[-1,1]$. For $d=3$, we compute the maximum approximation error  
\begin{equation}
    e_{\max}\left[f_i\right]=\max_{x\in X_{n_e}} \left\lvert f_i(x) - Q_{4}\left[f_i\right](x) \right\rvert, \quad i=1,2,3,4,
\end{equation}
for different values of $n_e$, where
\begin{equation}
    X_{n_e}=\left\{x_{i,e}=-1 + \frac{2}{n_e}i \, : \, i=0,\dots,n_e \right\}
\end{equation}
is the set of evaluation points. The results are shown in Table~\ref{tab1}.
\begin{table}
    \centering
    \begin{tabular}{|c|c|c|c|c|c|}\hline
        & $n_e=500$ & $ n_e=1000$ & $n_e=2000$ & $n_e=3000$ & $n_e=4000$\\ \hline
        $e_{\max}\left[f_1\right]$ & 1.1927e-09 	 &  2.1855e-07 	 & 2.8675e-03 	 & 6.0958e-02 & 2.7313e-01  \\
        $e_{\max}\left[f_2\right]$ & 1.8812e-10 	 &  4.3654e-08 	 &  5.7340e-04 	 & 1.2190e-02 &   5.4621e-02 \\
        $e_{\max}\left[f_3\right]$ & 7.2635e-12 	 & 1.3871e-08 	 & 1.8217e-04 	 & 3.8727e-03 & 1.7352e-02  \\
        $e_{\max}\left[f_4\right]$ & 6.0549e-09 & 9.1175e-05 &  9.3101e-03	 &  4.1357e-02 & 8.3835e-02\\\hline
    \end{tabular}
    \caption{Maximum approximation error $e_{\max}\left[f_i\right]$ for the functions $f_i$, $i=1,2,3,4$, computed using $n=1025$ nodes and by varying the set of evaluation points for different values of $n_e$.}
    \label{tab1}
\end{table}
From these results, we notice that by increasing the number of evaluation points $n_e$, we go across the singularity and then the maximum approximation error increases. This is a natural behavior, since we have no information on the function close to the singularity. In order to appreciate the quality of the approximation, in Figures~\ref{Qf1_rec},~\ref{Qf2_rec},~\ref{Qf3_rec},~\ref{Qf4_rec}, we plot the reconstruction of the functions through the quasi-interpolation operator~\eqref{Qoperator}.
More precisely, on the left we plot the functions $f_i$, $i=1,2,3,4$ and their quasi-interpolant, while on the right we zoom in on these plots to highlight the behavior of the approximant near the singularity.   
\begin{figure}
    \centering
 \includegraphics[width=0.5\linewidth]{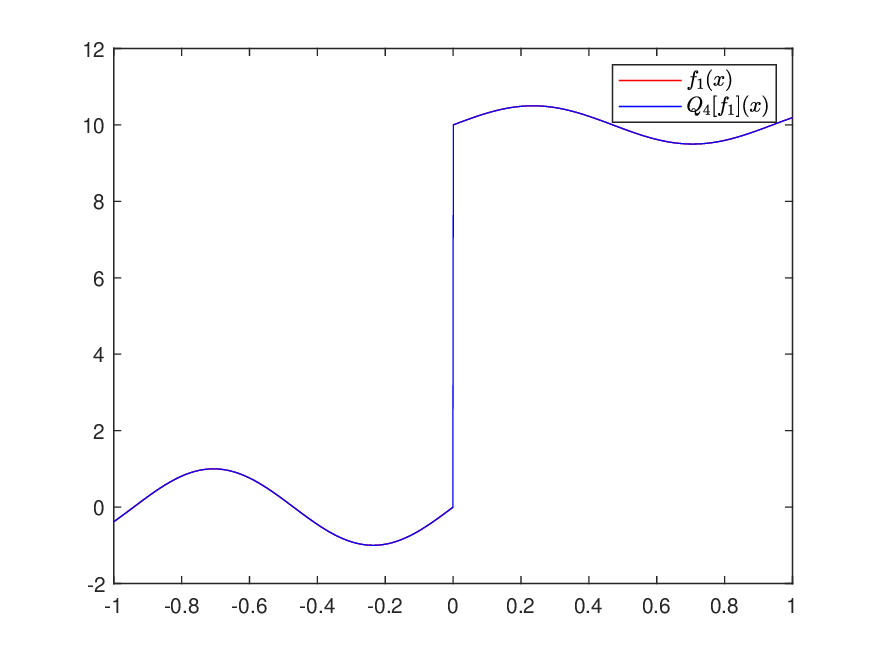}{}\includegraphics[width=0.5\linewidth]{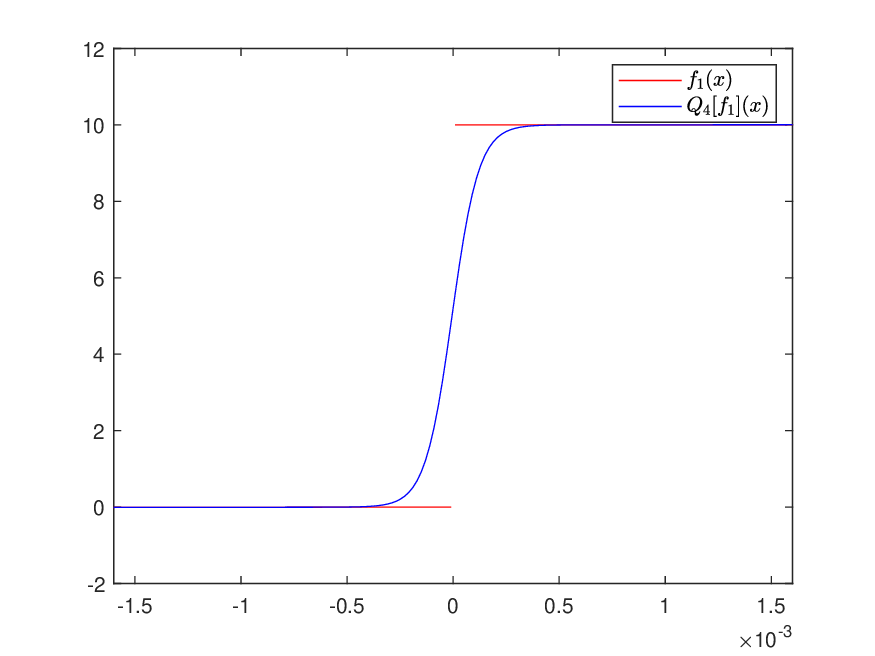}
    \caption{In the left, plot of the function $f_1$ (red) and its quasi-interpolant $Q_{4}[f_1](x)$ (blue) computed for $\mu=4$ and related to a grid of $n+1=1025$ equispaced nodes is displayed. In the right plot, to better appreciate the quality and the differentiability class of the approximation, we zoomed in to highlight the behavior of the approximant near the singularity $x=0$.}
    \label{Qf1_rec}
\end{figure}
\begin{figure}
    \centering
 \includegraphics[width=0.5\linewidth]{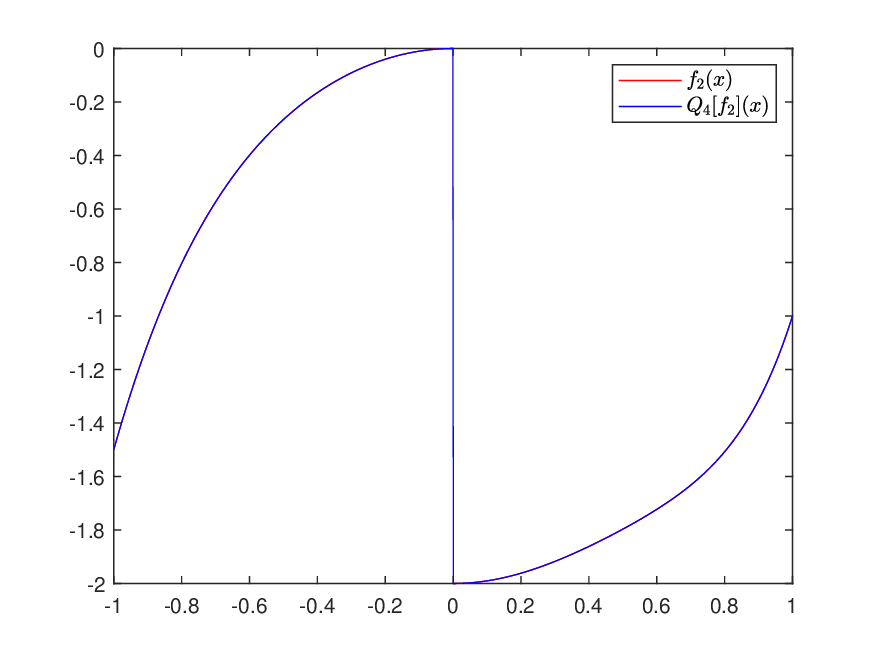}{}\includegraphics[width=0.5\linewidth]{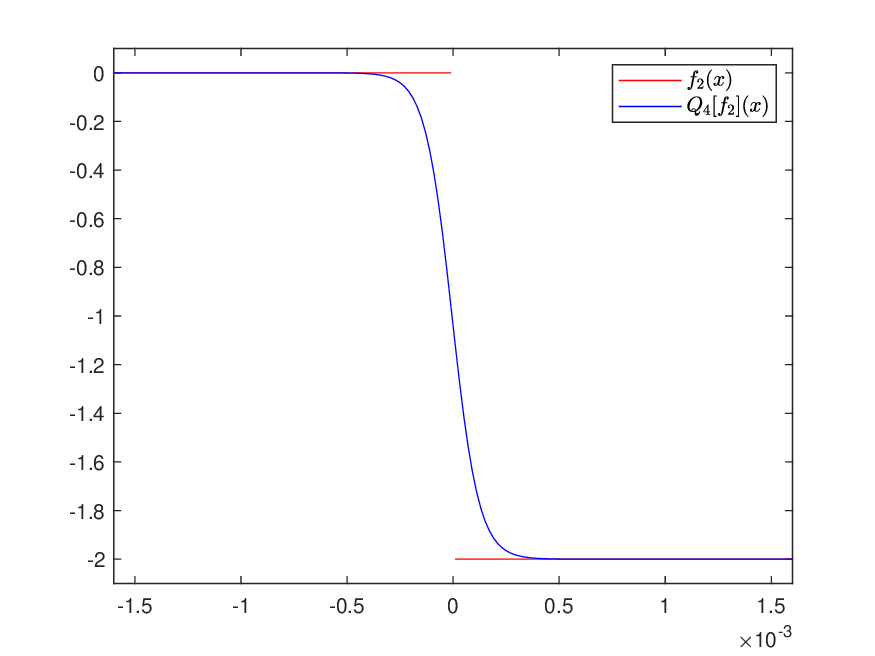}
    \caption{In the left, plot of the function $f_2$ (red) and its quasi-interpolant $Q_{4}[f_2](x)$ (blue) computed for $\mu=4$ and related to a grid of $n+1=1025$ equispaced nodes is displayed. In the right plot, to better appreciate the quality and the differentiability class of the approximation, we zoomed in to highlight the behavior of the approximant near the singularity $x=0$.}
    \label{Qf2_rec}
\end{figure}

\begin{figure}
    \centering
 \includegraphics[width=0.5\linewidth]{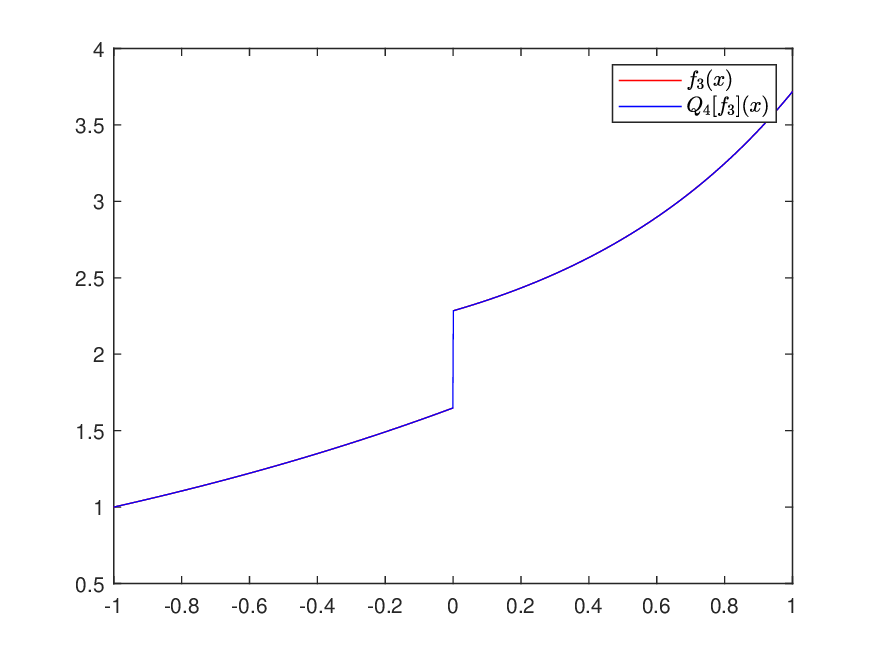}{}\includegraphics[width=0.5\linewidth]{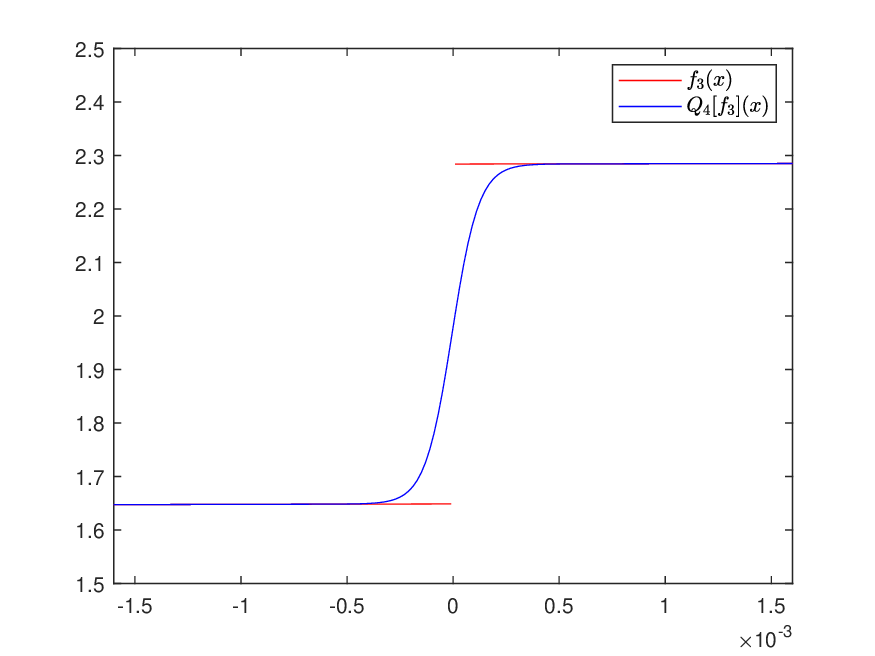}
    \caption{In the left, plot of the function $f_3$ (red) and its quasi-interpolant $Q_{4}[f_3](x)$ (blue) computed for $\mu=4$ and related to a grid of $n+1=1025$ equispaced nodes id displayed. In the right plot, to better appreciate the quality and the differentiability class of the approximation, we zoomed in to highlight the behavior of the approximant near the singularity $x=0$.}
    \label{Qf3_rec}
\end{figure}

\begin{figure}
    \centering
 \includegraphics[width=0.5\linewidth]{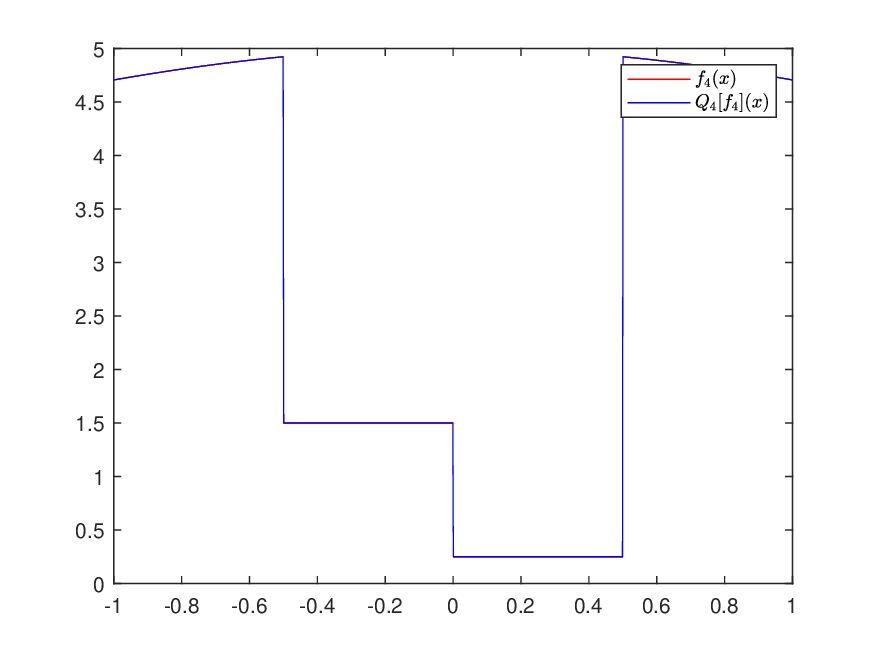}{}\includegraphics[width=0.5\linewidth]{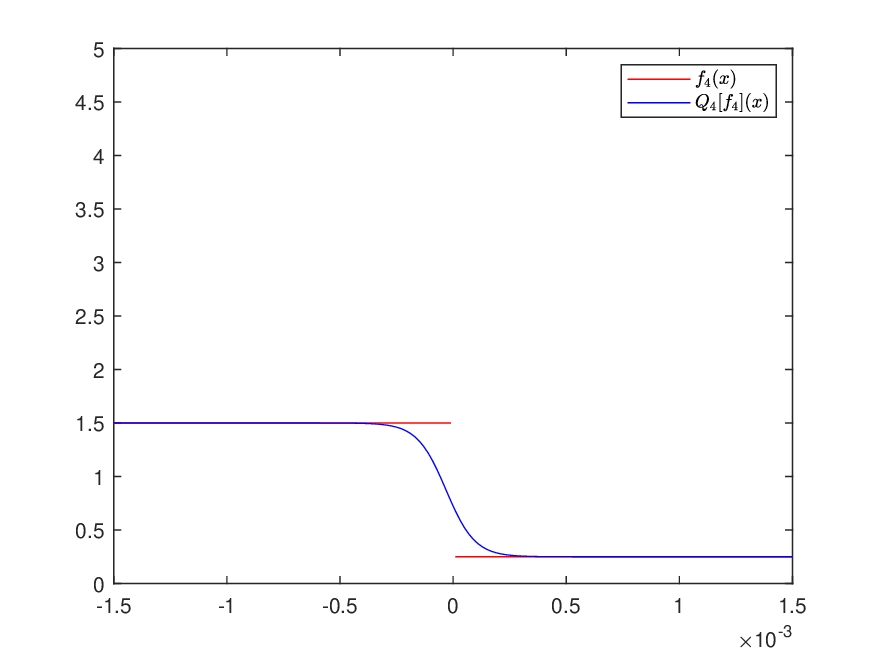}
    \caption{In the left, plot of the function $f_4$ (red) and its quasi-interpolant $Q_{4}[f_4](x)$ (blue) computed for $\mu=4$ and related to a grid of $n+1=1025$ equispaced nodes is displayed. In the right plot, to better appreciate the quality and the differentiability class of the approximation, we zoomed in to highlight the behavior of the approximant near the singularity $x=0$.}
    \label{Qf4_rec}
\end{figure}
\subsection{Numerical experiment $2$}
In the second class of numerical experiments, we assume to know the evaluations of the functions $f_i$, $i=1,2,3,4$, for different values of equispaced nodes ranging from $100$ to $1500$, with stepsize $200$, in $[-1,1]$. For different values of $d$, we analyze the trend of the error in $L_1$-norm
\begin{equation}
    e_1[f_i]= \int_{-1}^{1} \left\lvert f_i(x)-Q_4\left[f_i\right](x)\right\rvert \mathrm{d}x, \quad i=1,2,3,4,
\end{equation}
that is particularly well-suited for analyzing the error of a function with jump discontinuities. This is because it evaluates the total error across the whole domain, 
by offering a complete measure of deviation that takes 
effective account for sudden changes in the function while avoiding over-sensitivity to localized peaks. The results of these experiments are presented in Figures~\ref{L1f1f2},~\ref{L1f3f4}. In order to compute the $L_1$-norm we have used a Gaussian quadrature formula based on $20$ quadrature points and weights. 
We observe that, the trend of the errors $e_1\left[f_i\right]$, $i=1,2,3,4$, decrease as the number of equispaced nodes increases. Moreover, for each fixed number of nodes, the error also decreases as the value $d$ increases.

\begin{figure}
    \centering
    \includegraphics[width=0.49\linewidth]{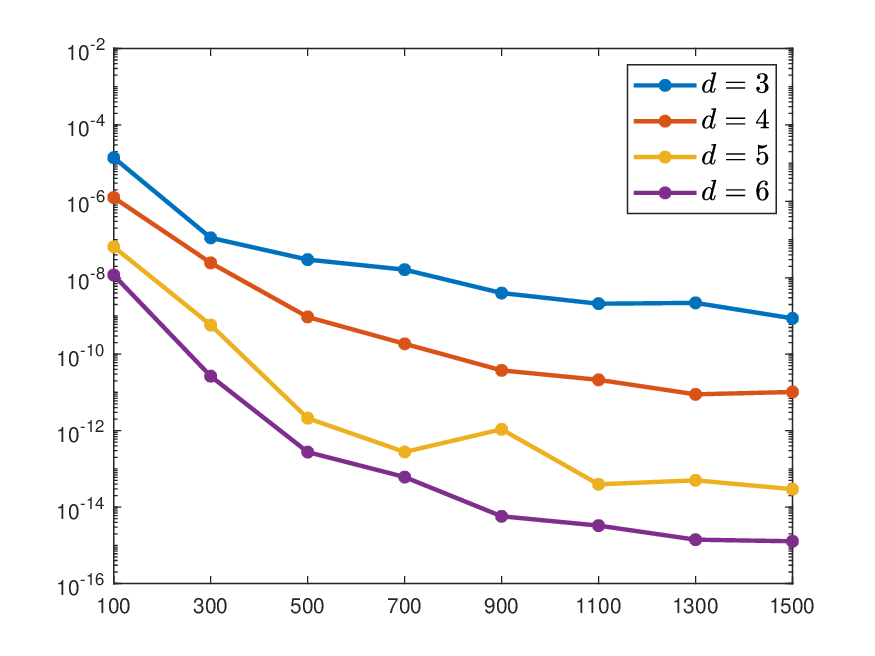}
    \includegraphics[width=0.49\linewidth]{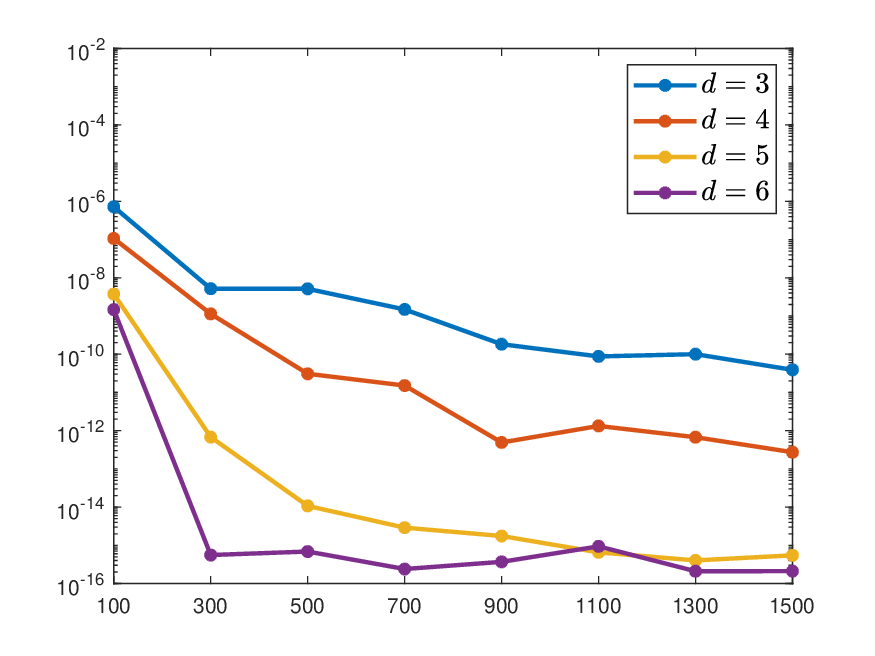}
    \caption{Semilog plot of the trend of the error in $L_1$-norm for the functions $f_1$ (left) and $f_2$, (right), produced by the quasi-interpolation operator~\eqref{Qoperator} with $\mu=4$, different values of $d$ and the number of equispaced nodes ranging from $100:200:1500$.}
    \label{L1f1f2}
\end{figure}

\begin{figure}
    \centering
    \includegraphics[width=0.49\linewidth]{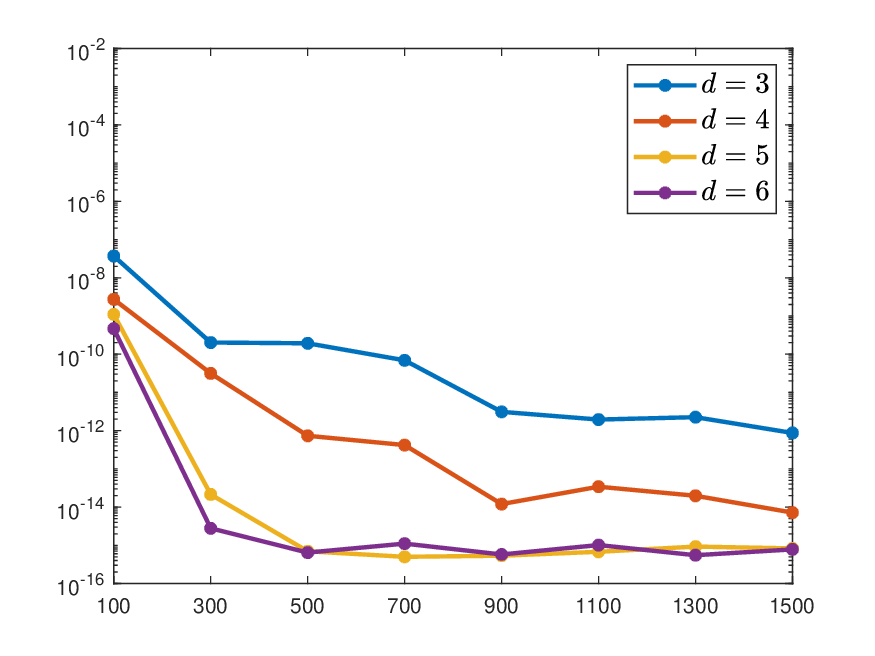}
    \includegraphics[width=0.49\linewidth]{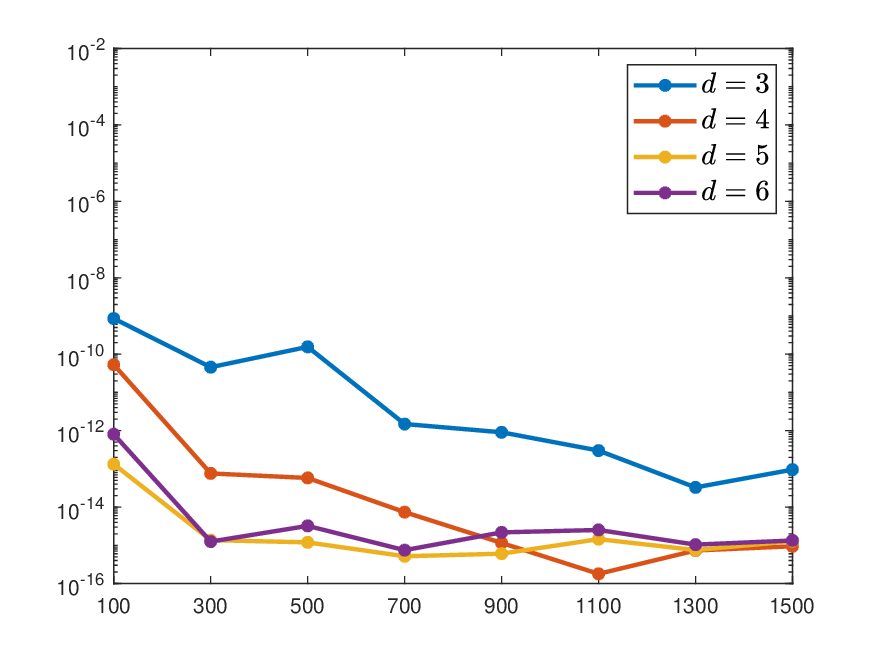}
    \caption{Semilog plot of the trend of the error in $L_1$-norm for the functions $f_3$ (left) and $f_4$, (right), produced by the quasi-interpolation operator~\eqref{Qoperator} with $\mu=4$, different values of $d$ and the number of equispaced nodes ranging from $100:200:1500$.}
    \label{L1f3f4}
\end{figure}
\subsection{Numerical experiment $3$}
In the third class of numerical experiments, we test the accuracy of our quasi-interpolation operator when the data are affected by additive white Gaussian noise. In this case,  by Remark~\ref{remarknoise}, we use the multinode Shepard functions blended with least squares polynomials of degree $d^{\prime}$. More precisely, we set $d^{\prime}=3$, $n+1=1025$, $d=6$ and we assume that the amplitude of the noise is $5\times10^{-1}$. In line with Subsection~\ref{subsection1}, we plot the reconstruction of the function $f_1$ as displayed in Figure~\ref{f1agwn}. From this plot, we observe that the oscillations produced by the quasi-interpolant do not reach the amplitude of the noise.   
\begin{figure}
    \centering
    \includegraphics[width=0.5\linewidth]{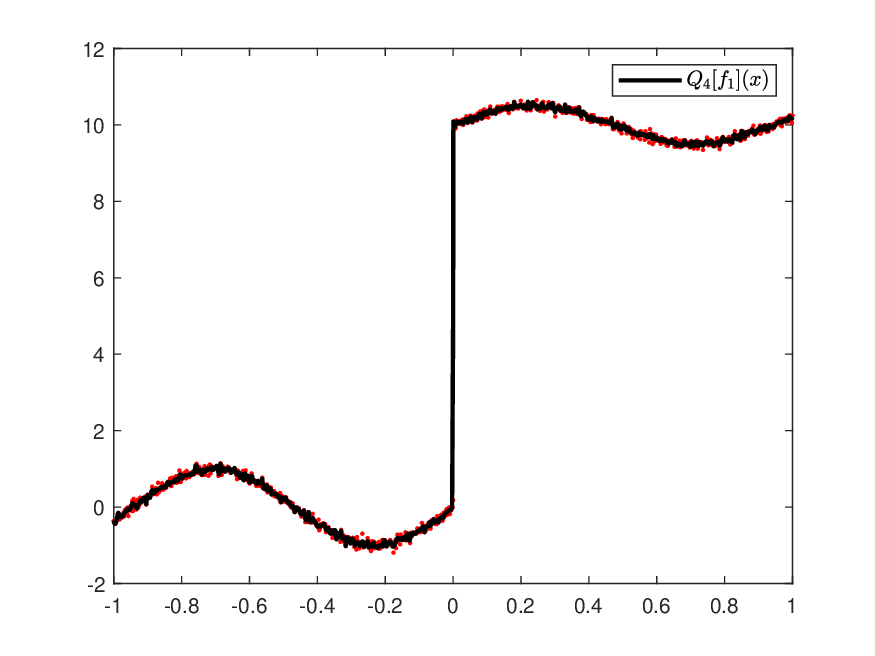}
    \caption{Plot of the quasi-interpolation operator $Q_4\left[f_1\right](x)$  (black lines) related to $n+1=1025$ equispaced nodes and the discretization of the associated sampled points (red points).}
    \label{f1agwn}
\end{figure}

\subsection{Numerical experiment $4$}
In this final set of numerical experiments, we consider the continuous functions 
\begin{equation}
    f_5(x) = \frac{1}{1+25x^2}, \quad f_6(x) = \cos(20\pi x), \quad x\in[-1,1].  
\end{equation}
We compute the trend of the maximum approximation error produced by the quasi-interpolation operator~\eqref{Qoperator}, for different values of $d$, while varying the stepsize of the grid of equispaced nodes. The results are reported in Figure~\ref{maxf5f6}. From these plots, we observe that the maximum approximation error decreases as the number of equispaced nodes increases. Moreover, the error becomes smaller as the degree $d$ increases. This behavior is consistent with the results obtained for discontinuous functions, confirming that the method performs well for both discontinuous and smooth functions.
\begin{figure}
    \centering
\includegraphics[width=0.49\linewidth]{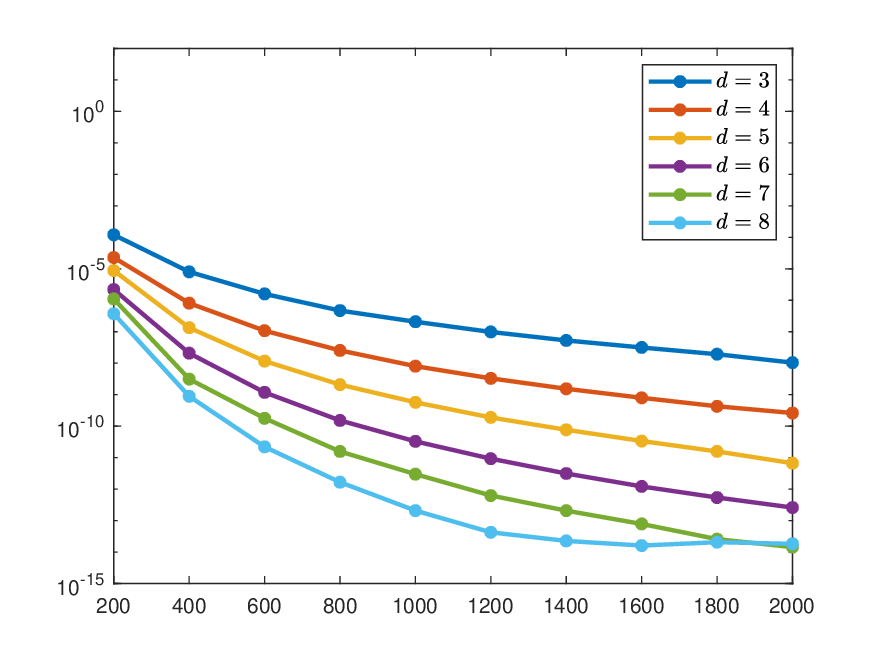}\includegraphics[width=0.49\linewidth]{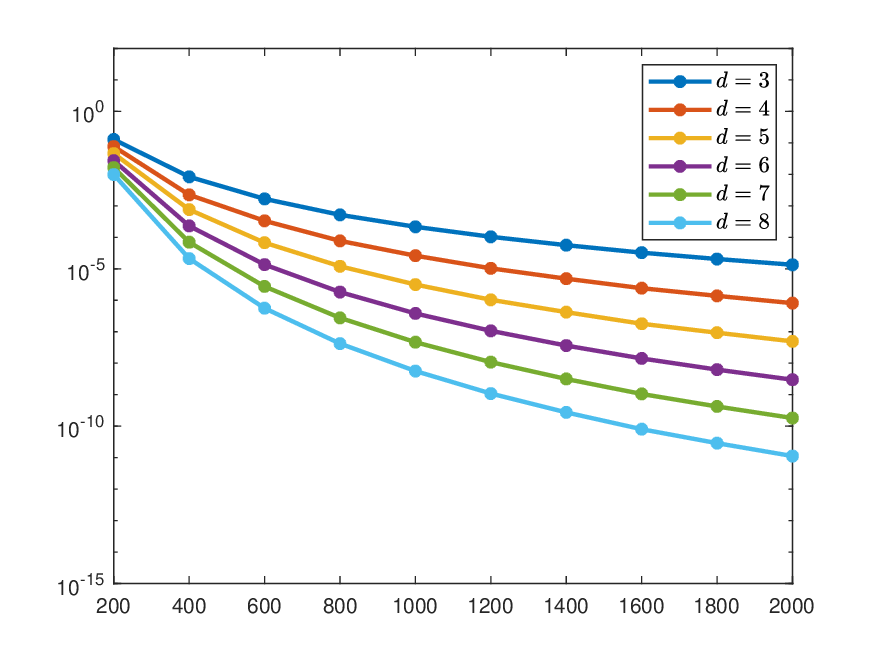}
    \caption{Semilog plots of the trend of the maximum approximation error for the functions $f_5$ (left) and $f_6$ (right), produced by the quasi-interpolation operator~\eqref{Qoperator} with $\mu=4$, for different values of $d$ and number of equispaced nodes ranging from $200$ to $2000$ with step $200$.}
    \label{maxf5f6}
\end{figure}

\section*{Acknowledgments}
The authors are grateful to the anonymous reviewers for carefully reading the manuscript and for their precise and helpful suggestions which allowed to improve the work.

This research has been achieved as part of RITA \textquotedblleft Research
 ITalian network on Approximation'' and as part of the UMI group \enquote{Teoria dell'Approssimazione
 e Applicazioni}. The research was supported by GNCS-INdAM 2025 project \lq\lq Polinomi, Splines e Funzioni Kernel: dall’Approssimazione Numerica al Software Open-Source\rq\rq.  The authors are members of the INdAM-GNCS Research group. The work of F. Nudo has been funded by the European Union – NextGenerationEU under the Italian National Recovery and Resilience Plan (PNRR), Mission 4, Component 2, Investment 1.2 \lq\lq Finanziamento di progetti presentati da giovani ricercatori\rq\rq,\ pursuant to MUR Decree No. 47/2025.

\bibliographystyle{elsarticle-num}
\bibliography{bibliography.bib}

\begin{thebibliography}{10}
\expandafter\ifx\csname url\endcsname\relax
  \def\url#1{\texttt{#1}}\fi
\expandafter\ifx\csname urlprefix\endcsname\relax\def\urlprefix{URL }\fi
\expandafter\ifx\csname href\endcsname\relax
  \def\href#1#2{#2} \def\path#1{#1}\fi

\bibitem{Wang:2004:QIW}
R.-H. Wang, J.-X. Wang, Quasi-interpolations with interpolation property, J.
  Comput. Appl. Math. 163 (2004) 253--257.

\bibitem{Buhmann:2022:QI}
M.~Buhmann, J.~J{\"a}ger, Quasi-interpolation, Cambridge University Press,
  2022.

\bibitem{Ma:2009:ATT}
L.~Ma, Z.~Wu, Approximation to the $k$-th derivatives by multiquadric
  quasi-interpolation method, J. Comput. Appl. Math. 231 (2009) 925--932.

\bibitem{Barrera:2013:ITA}
D.~Barrera, A.~Guessab, M.~J. Ib{\'a}{\~n}ez, O.~Nouisser, Increasing the
  approximation order of spline quasi-interpolants, J. Comput. Appl. Math. 252
  (2013) 27--39.

\bibitem{Barrera:2019:QIB}
D.~Barrera, C.~Dagnino, M.~J. Ib{\'a}{\~n}ez, S.~Remogna, Quasi-interpolation
  by ${C}^1$ quartic splines on type-1 triangulations, J. Comput. Appl. Math.
  349 (2019) 225--238.

\bibitem{Sun:2022:ACI}
Z.~Sun, W.~Gao, R.~Yang, A convergent iterated quasi-interpolation for periodic
  domain and its applications to surface {PDE}s, J. Sci. Comput. 93 (2022) 37.

\bibitem{Barrera:2022:ANC}
D.~Barrera, S.~Eddargani, A.~Lamnii, A novel construction of {B}-spline-like
  bases for a family of many knot spline spaces and their application to
  quasi-interpolation, J. Comput. Appl. Math. 404 (2022) 113761.

\bibitem{Ortmann:2024:HAQ}
M.~Ortmann, M.~Buhmann, High {A}ccuracy {Q}uasi-{I}nterpolation using a new
  class of generalized {M}ultiquadrics, J. Math. Anal. Appl. 538 (2024) 128359.

\bibitem{Arandiga:2024:EAW}
F.~Ar{\`a}ndiga, D.~Barrera, S.~Eddargani, {ENO} and {WENO} cubic
  quasi-interpolating splines in {B}ernstein--{B}{\'e}zier form, Math. Comput.
  Simul. 225 (2024) 513--527.

\bibitem{Schoenberg:1946:CTTA}
I.~J. Schoenberg, Contributions to the problem of approximation of equidistant
  data by analytic functions. {P}art {A}. {O}n the problem of smoothing or
  graduation. {A} first class of analytic approximation formulae, Quart. Appl.
  Math. 4 (1946) 45--99.

\bibitem{Schoenberg:1946:CTTB}
I.~J. Schoenberg, Contributions to the problem of approximation of equidistant
  data by analytic functions. {P}art {B}. {O}n the problem of osculatory
  interpolation. {A} second class of analytic approximation formulae, Quart.
  Appl. Math. 4 (1946) 112--141.

\bibitem{Arandiga:2024:NUW}
F.~Ar{\`a}ndiga, D.~Barrera, S.~Eddargani, M.~Ib{\'a}{\~n}ez, J.~Rold{\'a}n,
  Non-uniform {WENO}-based quasi-interpolating splines from the
  {B}ernstein--{B}{\'e}zier representation and applications, Math. Comput.
  Simul. 223 (2024) 158--170.

\bibitem{Cockburn:1998:ENO}
C.-W. Shu, Essentially non-oscillatory and weighted essentially non-oscillatory
  schemes for hyperbolic conservation laws, Springer Berlin Heidelberg, 1998,
  Ch.~4, pp. 325--432.

\bibitem{Gibbs:1898:FSS}
J.~W. Gibbs, Fourier's series, Nature 59 (1898) 200--200.

\bibitem{Buhmann:2024:NMF}
M.~Buhmann, J.~J{\"a}ger, J.~J{\'o}dar, M.~L. Rodr{\'\i}guez, New methods for
  quasi-interpolation approximations: {R}esolution of odd-degree singularities,
  Math. Comput. Simul. 223 (2024) 50--64.

\bibitem{guessab2011error}
A.~Guessab, M.~J. Ib{\'a}{\~n}ez, O.~Nouisser, Error analysis for a
  non-standard class of differential quasi-interpolants, Math. Comput. Simul.
  81 (2011) 2190--2200.

\bibitem{barrera2023lownew}
D.~Barrera, S.~Eddargani, M.~J. Ib{\'a}{\~n}ez, S.~Remogna, Low-degree spline
  quasi-interpolants in the {B}appl math comput.ernstein basis, Appl. Math.
  Comput. 457 (2023) 128150.

\bibitem{Amat:2021:OCS}
S.~Amat, J.~Ruiz, C.-W. Shu, J.~C. Trillo, On a class of splines free of
  {G}ibbs phenomenon, ESAIM: Math. Model. Numer. Anal. 55 (2021) S29--S64.

\bibitem{Amat:2022:ACO}
S.~Amat, D.~Levin, J.~Ruiz-Alvarez, J.~C. Trillo, D.~F. Y{\'a}{\~n}ez, A class
  of ${C}^2$ quasi-interpolating splines free of {G}ibbs phenomenon, Numer.
  Algorithms 91 (2022) 51--79.

\bibitem{Foucher:2009:QSQ}
F.~Foucher, P.~Sablonni{\`e}re, Quadratic spline quasi-interpolants and
  collocation methods, Math. Comput. Simul. 79 (2009) 3455--3465.

\bibitem{Gao:2014:AQI}
W.~Gao, Z.~Wu, A quasi-interpolation scheme for periodic data based on
  multiquadric trigonometric {B}-splines, J. Comput. Appl. Math. 271 (2014)
  20--30.

\bibitem{DellAccio:2018:ROF}
F.~Dell’Accio, F.~Di~Tommaso, K.~Hormann, Reconstruction of a function from
  {H}ermite--{B}irkhoff data, Appl. Math. Comput. 318 (2018) 51--69.

\bibitem{DellAccio:2019:RCM}
F.~Dell'Accio, F.~Di~Tommaso, Rate of convergence of multinode {S}hepard
  operators, Dolomites Res. Notes Approx. 12 (2019) 1--6.

\bibitem{DellAccio:2021:NDS}
F.~Dell’Accio, F.~Di~Tommaso, N.~Siar, M.~Vianello, Numerical differentiation
  on scattered data through multivariate polynomial interpolation, BIT Numer.
  Math. 62 (2022) 773–801.

\bibitem{DellAccio:2021:SPE}
F.~Dell'Accio, F.~Di~Tommaso, O.~Nouisser, N.~Siar, Solving {P}oisson equation
  with {D}irichlet conditions through multinode {S}hepard operators, Comput.
  Math. Appl. 98 (2021) 254--260.

\bibitem{DellAccio:2023:OTI}
F.~Dell'Accio, F.~Di~Tommaso, A.~Guessab, F.~Nudo, On the improvement of the
  triangular {S}hepard method by non conformal polynomial elements, Appl.
  Numer. Math. 184 (2023) 446--460.

\bibitem{Cavoretto:2024:NCS}
R.~Cavoretto, A.~De~Rossi, F.~Dell’Accio, F.~Di~Tommaso, N.~Siar,
  A.~Sommariva, M.~Vianello, Numerical cubature on scattered data by adaptive
  interpolation, J. Comput. Appl. Math. 444 (2024) 115793.

\bibitem{DellAccio:2024:MSM}
F.~Dell’Accio, F.~Di~Tommaso, E.~Francomano, Multinode {S}hepard method for
  two-dimensional elliptic boundary problems on different shaped domains, J.
  Comput. Appl. Math. 448 (2024) 115896.

\bibitem{Gautschi:2011:NA}
W.~Gautschi, Numerical analysis, Birkhäuser, 2011.

\bibitem{Costarelli:2017:ADS}
D.~Costarelli, A.~M. Minotti, G.~Vinti, Approximation of discontinuous signals
  by sampling {K}antorovich series, J. Math. Anal. Appl. 450 (2017) 1083--1103.

\end{thebibliography}

\end{document}